\documentclass[11pt,leqno]{article}
\usepackage[cm]{fullpage}
\usepackage{enumitem}
\usepackage{amscd}
\usepackage{ amssymb }
\usepackage{amsmath}
\usepackage{ esint }
\usepackage{systeme,mathtools}
\usepackage{amsfonts}
\usepackage{amsthm}
\usepackage{tikz-cd}
\usepackage{stmaryrd}
\usepackage{hyperref}
\usepackage{cleveref} 
\usepackage{lipsum}
\usepackage{chngcntr}

\DeclareMathOperator{\divg}{div}
\DeclareMathOperator{\im}{Im}
\DeclareMathOperator{\tr}{tr}
\DeclareMathOperator{\Aut}{Aut}
\DeclareMathOperator{\End}{End}
\DeclareMathOperator{\Diff}{Diff}
\DeclareMathOperator{\GDiff}{GDiff}
\DeclareMathOperator{\Isom}{Isom}
\DeclareMathOperator{\Rc}{Rc}

\begin{document}

\theoremstyle{definition}
\newtheorem{claim}{Claim}
\theoremstyle{plain}
\newtheorem{proposition}{Proposition}[section]
\newtheorem{theorem}[proposition]{Theorem}
\newtheorem{lemma}[proposition]{Lemma}
\newtheorem{corollary}[proposition]{Corollary}
\theoremstyle{definition}
\newtheorem{defn}[proposition]{Definition}
\theoremstyle{remark}
\newtheorem{remark}[proposition]{Remark}
\theoremstyle{definition}
\newtheorem{example}[proposition]{Example}
\theoremstyle{definition}
\newtheorem*{Motivation}{Motivation}

\title{\textbf{The Stability of Generalized Ricci Solitons}}
\author{KUAN-HUI LEE}
\date{}
\maketitle

\begin{abstract}
    In \cite{GRF, P}, it was shown that the generalized Ricci flow is the gradient flow of a functional $\lambda$ generalizing Perelman's $\lambda$ functional for Ricci flow. In this work, we further computed the second variation formula and proved that a Bismut-flat, Einstein manifold is linearly stable under some curvature assumptions. In the last part of this paper, I proved that dynamical stability and linear stability are equivalent on a steady gradient generalized Ricci soliton $(g,H,f)$. This generalizes the results in  \cite{H2,Kroencke2014,K4,DG,MR2219268}. 
\end{abstract}

\setlength\arraycolsep{2pt}

\section{Introduction}

Suppose that $M$ is a smooth manifold and $H_0$ is a closed 3-form. Let $\mathcal{M}$ denote the space of Riemannian metrics and $\Omega^2$ is the space of two forms. We say that $(g_t,b_t)\in \mathcal{M}\times \Omega^2$ is a generalized Ricci flow (we will abbreviate it as GRF later on) if 
\begin{align}
    \frac{\partial}{\partial t}g=-2\Rc+\frac{1}{2}H^2, \quad
   \frac{\partial}{\partial t}b=-d^*H \quad \text{ where $H=H_0+db$}. \label{1}
\end{align}

This parabolic flow was written in \cite{P}, \cite{J} and it is a generalization of the Ricci flow. A few years ago, Perelman \cite{Pe} discovered that the Ricci flow can be viewed as a gradient flow of the Perelman $\mathcal{F}$-functional. His ideas can also be applied to the generalized Ricci flow. We define the generalized Einstein--Hilbert functional 
\begin{align}
    \mathcal{F}\colon\nonumber&\quad \Gamma(S^2M)\times \Omega^3\times C^\infty (M)\to \mathbb{R}
    \\& \quad (g,H,f)\longmapsto \int_M (R-\frac{1}{12}|H|^2+|\nabla f|^2)e^{-f}dV_g \label{2}
\end{align}
and  
\begin{align}
    \lambda(g,H)\coloneqq\inf\Big\{\mathcal{F}(g,H,f)\big| f\in C^\infty(M),\,\int_M e^{-f}dV_g=1\Big\}. \label{3}
\end{align}

One can see that $\lambda(g,H)$ can be achieved by some $f$ uniquely, i.e., $\lambda(g,H)=\mathcal{F}(g,H,f)$ and $\lambda$ is the first eigenvalue of the Schrödinger operator $-4\triangle+R-\frac{1}{12}|H|^2$. In \cite{P}, it was shown that $\lambda$ is monotone increasing under the generalized Ricci flow and the critical points of $\lambda$ are the steady gradient generalized Ricci solitons. (See Section 3.1 for more details.) 

\begin{defn}\label{D1}
Given $(g,H)$ a pair of smooth metric and closed 3-form, we say that $(g,H)$ is a \emph{steady generalized Ricci soliton} if there exists a smooth vector field $X$ such that
\begin{align}
    0=\Rc-\frac{1}{4}H^2+\frac{1}{2}L_Xg, \quad 0=d^*_gH+i_XH. \label{4}
\end{align}
Furthermore, we say that the soliton is \emph{gradient} if there exists a smooth function $f$ such that $X=\nabla f$. $(g,H)$ is called a \emph{generalized Einstein metric} if $X=0$.

\end{defn}  

The most important example of the steady generalized Ricci soliton is the work on $S^3$. Given a standard unit sphere metric $g_{S^3}$. By taking $H_{S^3}=2dV_{g_{S^3}}$, we get
\begin{align*}
    \Rc_{g_{S^3}}=\frac{1}{4}H^2_{S^3}, \quad d^*H_{S^3} =0
\end{align*}
which implies that $(g_{S^3},H_{S^3} )$ is a generalized Einstein metric. In \Cref{C3}, we see that any three-dimensional generalized Einstein manifold is a quotient of $S^3$. Moreover, in \cite{classification}, Streets showed that
there exists a family of generalized steady Ricci solitons on $S^3$ which forms a smooth family with the round $S^3$. Thus, $S^3$ is not rigid.

Next, let us compute the second variation formula. 
\begin{theorem}\label{T1}
Suppose $(g_t,H_t)$ is a one parameter family such that 
\begin{align*}
    &\frac{\partial}{\partial t}\Big|_{t=0}g_t =h, \quad  g_0=g,
    \\&  \frac{\partial}{\partial t}\Big|_{t=0}H_t =dK, \quad  H_0=H
\end{align*}
and $f_t$ satisfies $\lambda(g_t,H_t)=\mathcal{F}(g_t,H_t,f_t)$.
The second variation of $\lambda$ on a compact steady gradient generalized Ricci soliton $(M,g,H,f)$ is given by
\begin{align}
     \nonumber\frac{d^2}{dt^2}\Big|_{t=0}\lambda&=\int_M \Big\langle \frac{1}{2}\triangle_f h+(\mathring{R} h)+\divg^*_f \divg_fh,h \Big\rangle e^{-f}dV_g
   \nonumber \\&\kern2em + \int_M \Big[-\frac{1}{4}h_{ij}h_{ik}H^2_{jk}-\frac{1}{2}h_{ij}h_{ac}H_{iab}H_{jcb}+h_{ij}(dK)_{iab}H_{jab}\Big]e^{-f}dV_g
      \nonumber\\&\kern2em -\frac{1}{6}\int_M |dK|^2 e^{-f}dV_g-\int_M \frac{1}{2}|\nabla v_{h,K}|^2e^{-f}dV_g. \label{5}
\end{align} 
Here $v_{h, K}$ is the unique (up to a constant) solution of
\begin{align*}
    \triangle_f(v_{h,K})=\divg_f \divg_fh-\frac{1}{6}\langle dK,H \rangle. \quad \text{(Check the notation of $\divg_f$ and $\triangle_f$ at (\ref{18})).}
\end{align*}

\end{theorem}
Our second variation formula can be given by a self-adjoint operator $\mathcal{N}$ which is defined by
\begin{align*}
    \nonumber\mathcal{N}\colon \quad &\Gamma(S^2M)\times \Omega^2\times C^\infty(M)  \longrightarrow \Gamma(S^2M)\times \Omega^2 \times C^\infty(M)
    \\&  \quad (h,K,v_{h,K})\longmapsto (A(h)+C(K),B(K)+D(h),\frac{\triangle_f v_{h,K}}{2}).
\end{align*}
Here  $A: \Gamma(S^2M)\to \Gamma(S^2M)$, $B: \Omega^2\to \Omega^2$, $C:\Omega^2\to \Gamma(S^2M)$ and $D:\Gamma(S^2M)\to \Omega^2$ are
\begin{align*}
    &A(h)_{ij}\coloneqq \frac{1}{2}\triangle_f h_{ij}+(\mathring{R} h)_{ij}+(\divg^*_f \divg_fh)_{ij}-\frac{1}{8}h_{ik}H^2_{jk}-\frac{1}{8}h_{jk}H^2_{ik}-\frac{1}{2}h_{ac}H_{iab}H_{jcb},
    \\& B(K)_{ij}\coloneqq-\frac{1}{2}(d^*dK)_{ij}-\frac{1}{2}(dK)_{lij}\nabla_l f,
    \\& C(K)_{ij}\coloneqq\frac{1}{4}((dK)_{iab}H_{jab}+H_{iab}(dK)_{jab}),
    \\& D(h)_{ij}\coloneqq\frac{1}{2}(-h_{ab}\nabla_a H_{bij}-(\divg h)_b H_{bij}+\nabla_a h_{ib}H_{baj}+\nabla_a h_{jb}H_{bia}+h_{ab}H_{bij}\nabla_a f).
\end{align*}

\begin{corollary}
Define $\mathcal{N}$ as above, the second variation of $\lambda$ on a compact steady gradient generalized Ricci soliton $(M^n,g,H,f)$ is given by
\begin{align*}
   \frac{d^2}{dt^2}\Big|_{t=0}\lambda=\int_{M} \langle \mathcal{N}(h,K,v_{h,K}),(h,K,v_{h,K}) \rangle e^{-f}dV_g=\Big(\mathcal{N}(h,K,v_{h,K}),(h,K,v_{h,K})\Big)_{f}
\end{align*}
where $\Big(\cdot ,\cdot\Big)_{f}$ is a $L^2$ inner product given by \emph{(\ref{25})}.
\end{corollary}

The goal of this work is to discuss the stability behavior of compact steady gradient generalized Ricci solitons by using the second variation formula (\ref{5}). Let $(M^n,g,H,f)$ be a steady generalized gradient Ricci soliton with its corresponding generalized metric $\mathcal{G}=\mathcal{G}(g,0)$ and a background $3$-form $H$. We have two different concepts about stability.
\\
\\\textbf{(1) Dynamical stability:  }We say that a steady gradient generalized Ricci soliton $(g,H,f)$ is \emph{dynamically stable} if for any neighborhood $\mathcal{U}$ of $(g,0)$ in $\mathcal{M}\times \Omega^2$, there exists a smaller neighborhood $\mathcal{V}\subset\mathcal{U}$ such that the generalized Ricci flow $(g_t,b_t)$ starting in $\mathcal{V}$ stays in $\mathcal{U}$ for all $t\geq 0$ and converges to a critical point $(g_\infty,b_\infty)$ of $\lambda$ with $\lambda(g,H)=\lambda(g_\infty,H_\infty)$ where $H_\infty=H+db_\infty$.
\\
\\\textbf{(2) Linear stability:  }We say that a steady gradient generalized Ricci soliton $(g,H,f)$ is \emph{linearly stable} if 
\begin{align*}
    \frac{d^2}{dt^2}\lambda(h,K)\leq 0 \quad \text{ for all variation $(h,K)$ at $(g,H,f)$}.
\end{align*}

To discuss the linear stability, we first observe that when $H=0$ the steady gradient generalized Ricci solitons are Ricci flat metrics. They are stable if and only if $\triangle_L$ is nonpositive (which is the same as the Einstein--Hilbert functional case, see \cite{B} and \cite{P} for detail). Next, we try to derive a similar result for generalized Einstein manifolds. However, it is not easy to find a general result without motivation. In Section 4, we see that any compact semisimple Lie group is Bismut-flat and Einstein, which motivates us to consider a Bismut-flat, Einstein manifold $(M^n,g,H)$. The following is our result.
\begin{theorem}\label{T2}
Suppose $(M^n,g,H)$ is a Bismut-flat, Einstein manifold. $(M^n,g,H)$ is linearly stable if $\triangle_G|_{TT_g}$ is negative semidefinite where
\begin{align*}
     \triangle_G(h)\coloneqq \frac{1}{2}\triangle h+3\mathring{R}(h)+\frac{1}{2}(\Rc\circ h+h\circ \Rc)
\end{align*}
and $TT_g\coloneqq\{h: \tr_gh=0,\divg h=0\}$.
\end{theorem}
\begin{remark}
It is not true that linearly stable implies $\triangle_G|_{TT_g}$ is negative semidefinite. We will give a reason in \Cref{R11}. 
\end{remark}

Due to the diffeomorphism invariance of $\lambda$, to test the linear stability we may only focus on the non-trivial variations. Recall that in the Einstein--Hilbert functional case, one can construct an Ebin's slice $\mathcal{S}_g$ for any metric $g$ (see \cite{MR0267604}) and non-trivial variations come from the tangent space of the slice $T_g\mathcal{S}_g$. In \cite{Rubio_2019}, Rubio and Tipler generalized Ebin's work on the generalized manifolds so that we have a generalized slice $\mathcal{S}_{\mathcal{G}}$ for any generalized metric $\mathcal{G}$. However, we can not apply their result directly since we 
use different $L^2$ inner products. To apply their result to our case, we need to modify their proof and derive the following.

\begin{theorem}\label{T}
Let $\mathcal{G}$ be a generalized metric on an exact Courant algebroid $E$ and suppose $f$ is $\Isom(\mathcal{G})$ invariant, then there exists an ILH submanifold $S^f_\mathcal{G}$ in the space of generalized metric $\mathcal{GM}$ such that 
\\(1)  $\forall F\in \Isom_H(\mathcal{G}),\, F\cdot S^f_\mathcal{G}=S^f_\mathcal{G}$.
\\(2) $ \forall F\in \GDiff_H, \text{if $(F\cdot S^f_\mathcal{G})\cap S^f_\mathcal{G}=\emptyset $ then }F\in \Isom_H(\mathcal{G})$.
\\(3) There exists a local cross section $\chi$ of the map $F\longmapsto \rho_{\mathcal{GM}}(F,\mathcal{G})$ on a neighborhood $\mathcal{U}$ of $\mathcal{G}$ in $\GDiff_H\cdot\mathcal{G}$ such that the map from $\mathcal{U}\times S^f_{\mathcal{G}}\longrightarrow \mathcal{GM}$ given by $(V_1,V_2)\longmapsto \rho_{\mathcal{GM}}(\chi(V_1),V_2)$ is a homeomorphism onto its image.   
\end{theorem}

Suppose $(g,H,f)$ is a steady gradient generalized Ricci soliton. By taking the background 3-form as $H$, we may consider a generalized metric $\mathcal{G}=\mathcal{G}(g,0)$. Since $f$ is $\Isom(\mathcal{G})$ invariant, we are able to apply \Cref{T}. In particular, in the generalized Einstein case, we decompose the variation space as 
\begin{align*}
    \Gamma(S^2M)\times \Omega^2=(\mathcal{V}+\mathcal{V}_1)\oplus(\mathcal{V}^\perp\cap \mathcal{V}_1^\perp)
\end{align*}
where subspaces $\mathcal{V}$  and $\mathcal{V}_1$ are given in (\ref{27}) and (\ref{29}). To finish the proof of \Cref{T2}, we need to prove that the second variation is nonpositive on each part and cross terms. When $n=3$, the main difficulty lies in the $\mathcal{V}_1$ part. To prove that the second variation is nonpositive on $\mathcal{V}_1$, we need to use the eigenfunction of Laplacian. On the other hand, in the higher dimensional case, the second variation formula is more complicated. Fortunately, it will be simpler if we consider Bismut flat, Einstein manifolds. First, the second variation is nonpositive on $\mathcal{V}_1$ by the Cauchy--Schwarz inequality. Second, the Bochner technique suggests to us that the second variation is also nonpositive on the subspace $\mathcal{V}^\perp\cap \mathcal{V}_1^\perp$. Last, following these two observations, we prove our result by choosing a proper decomposition of $dK$.

By direct computation, one can deduce that on 3-dimensional generalized Einstein manifold $(M,g,H)$, the operator $\triangle_G$ is negative semidefinite on $TT_g$ so we have the following corollary.
\begin{corollary}
Any 3-dimensional generalized Einstein manifold $(M,g,H)$ is linearly stable.
\end{corollary}

The last topic in this work is dynamical stability. One of the main topics about stability is that one expects that dynamical stability and linear stability are equivalent. Many people tried to prove this result in some cases. To our understanding so far, in the Ricci flow case, these two stability conditions are equivalent on 
compact Ricci flat manifolds \cite{H1,H2,MR2219268}
, Einstein manifolds \cite{Kroencke2014,K4} and gradient shrinking Ricci solitons  \cite{K3}. For generalized geometry, Raffero and Vezzoni \cite{DG} proved that dynamical stability and linear stability are equivalent for the generalized Ricci flow under the $H=0$ assumption. In this work, we would like to generalize the case to the steady gradient generalized Ricci solitons.

The key step of the proof is to use the Lojasiwicz--Simon inequality to control the growth of $\lambda$. In \cite{MR3211041} Theorem 6.3, Colding and Minicozzi developed a general version of the Lojasiwicz--Simon inequality. To satisfy the conditions of their theorem, we may use the \Cref{T} to show that our second variation operator is Fredholm. Finally, we can follow a similar argument of the dynamical stability of Einstein manifolds cases to prove the following.

\begin{theorem}\label{T3}
Let $(M^n,g_c,H_c,f_c)$ be a steady generalized gradient Ricci soliton with its corresponding generalized metric $\mathcal{G}_c=\mathcal{G}_c(g_c,0)$ and a background $3$-form $H_c$. Suppose $(M^n,g_c,H_c,f_c)$ is linearly stable and $k\geq 3$. Then, for every $C^k$-neighborhood $\mathcal{U}=B^k_\epsilon$ of $(g_c,0)$, there exists some $C^{k+2}$-neighborhood $\mathcal{V}$ such that for any GRF $(g(t),b(t))$ starting at $\mathcal{V}$, there exists a family of diffeomorphisms $\{\varphi_t\}$ such that
\begin{align*}
    \|(\varphi^*_tg_t,\varphi^*_tH_t)-(g_c,H_c)\|_{C^{k-1}_{g_c}}<\epsilon \quad \text{ for all $t$}
\end{align*}
and $(\varphi^*_t g_t,\varphi^*_tH_t)$ converges to $(g_\infty,H_\infty)$ at a polynomial rate with $\lambda(g_c,H_c)=\lambda(g_\infty,H_\infty)$.
\end{theorem}

Note that this theorem tells us that up to diffeomorphisms our solution of $(g_t,H_t)$ will converge at some point $(g_\infty,H_\infty)$ and this point may not be $(g_0,H_0)$. For example, we prove that the $S^3$ is linearly stable so \Cref{T3} only suggests that solutions of GRF near $(g_{S^3},H_{S^3})$ will converge to a steady gradient generalized Ricci soliton near $(g_{S^3},H_{S^3})$.

Besides the stability result, a similar argument also deduces that if $\lambda$ is not linearly stable, then it should be dynamically unstable.

\begin{theorem}\label{T4}
Let $(M^n,g_c,H_c,f_c)$ be a steady generalized gradient Ricci soliton with its corresponding generalized metric $\mathcal{G}_c=\mathcal{G}_c(g_c,0)$ and a background $3$-form $H_c$. Suppose $(g_c,H_c)$ is not a local maximum of $\lambda$ then there exists a GRF $(g_t,H_t)$ and a family of diffeomorphism $\{\varphi_t\}$, $t\in (-\infty,0]$ such that  
\begin{align*}
    (\varphi^*_t g_t,\varphi^*_t H_t)\longrightarrow (g_c,H_c) \text{  as $t\to-\infty$.}
\end{align*}
\end{theorem}

The layout of this paper is as follows: in Section 2, we will mention some preliminaries regarding Courant algebroids, the generalized slice theorem, and the generalized Ricci flow. In Section 3, we will discuss the variation formulas including the first variation formula and the second variation formula. In Section 4, we will prove the main results relative to linear stability. In Section 5, we will show that the dynamical stability and the linear stability are equivalent on the steady gradient generalized Ricci solitons.
\\

\textbf{Acknowledgements:} This work is written when the author is a second-year math Ph.D. student at the University of California-Irvine. I am grateful to my advisor Jeffrey D.~Streets for his helpful advice. His suggestions play an important role in this work. 

\section{Generalized Geometry}

\subsection{Courant Algebroids and Generalized Metrics}
\begin{defn}\label{D2}
A \emph{Courant algebroid} is a vector bundle $E\longrightarrow M$ with a nondegenerate bilinear form $\langle \cdot,\cdot\rangle$, a bracket $[\cdot,\cdot]$ on $\Gamma(E)$ and a bundle map $\pi:E\longrightarrow TM$ satisfies that for all $a,b,c\in\Gamma(E)$, $f\in C^\infty(M)$,
\begin{align*}
    &\text{(1) \quad} [a,[b,c]]=[[a,b],c]+[b,[a,c]].
    \\&\text{(2) \quad} \pi[a,b]=[\pi(a),\pi(b)].
     \\&\text{(3) \quad} [a,fb]=f[a,b]+\pi(a)fb.
      \\&\text{(4) \quad} \pi(a)\langle b,c\rangle=\langle [a,b],c \rangle+\langle a,[b,c] \rangle.
       \\&\text{(5) \quad}[a,b]+[b,a]=\mathcal{D}\langle a,b \rangle 
\end{align*}
where $\mathcal{D}: C^\infty(M)\to \Gamma(E)$ is given by $\mathcal{D}(\phi)\coloneqq\pi^*(d\phi)$. A Courant algebroid $E$ is called \emph{exact} if we have the following exact sequence of vector bundles 
\[  \begin{tikzcd}
  0 \arrow[r] & T^*M  \arrow[r, "\pi^*"] & E \arrow[r, "\pi"] & TM  \arrow[r] & 0 .
\end{tikzcd}
\]
\end{defn}

The most common and important example of Courant algebroids is $TM\oplus T^*M$. In this case, we define a nondegenerate bilinear form $\langle\cdot,\cdot\rangle$ and a bracket $[\cdot,\cdot]$ on $TM\oplus T^*M$ by
\begin{align*}
    \langle X+\xi,Y+\eta \rangle&\coloneqq\frac{1}{2}(\xi(Y)+\eta(X)),
    \\ [X+\xi,Y+\eta]_H&\coloneqq[X,Y]+L_X\eta-i_Y d\xi+i_Yi_XH
\end{align*}
where $X,Y\in TM$,  $\xi,\eta\in T^*M$ and $H$ is a 3-form. Define $\pi$ to be the standard projection, one can check directly that $(TM\oplus T^*M)_H\coloneqq(TM\oplus T^*M,\langle\cdot,\cdot\rangle,[\cdot,\cdot]_H,\pi)$ satisfies Courant algebroid conditions when $H$ is a closed 3-form. 

In the following, we will always assume that $E$ is an exact Courant algebroid. It is proved in \cite{GRF} Proposition 2.10, that every exact Courant algebroid with an isotropic splitting is isomorphic to $(TM\oplus T^*M)_H$ with some proper $H$. The complete statement is as follows.
\begin{proposition}[\cite{GRF} Proposition 2.10]\label{P1}
Given an exact Courant algebroid $E$ with a isotropic splitting $\sigma$, $E\cong_\sigma (TM\oplus T^*M)_H$ via the isomorphism $F:TM\oplus T^*M\longrightarrow E$ defined by
\begin{align*}
    F(X+\xi)= \sigma X+\frac{1}{2}\pi^*\xi \quad X\in TM,\xi\in T^*M. 
\end{align*}
Here the closed 3-from $H$ is given by
\begin{align}
    H(X,Y,Z)=2\langle [\sigma X,\sigma Y], \sigma Z \rangle \quad X,Y,Z \in TM. \label{6}
\end{align}

\end{proposition}

Here, we introduce the definition of the automorphism group of Courant algebroid which we will study more in the next section.

\begin{defn}\label{D4}
Let $E$ be an exact Courant algebroid. The \emph{automorphism group} $\Aut(E)$ of $E$ is a pair $(f,F)$ where $f\in\Diff(M)$ and $F\colon E\to E$ is a bundle map such that for all $u,v\in \Gamma(E)$
\begin{align*}
     &\text{(1) \quad} \langle Fu,Fv \rangle=f_{*}\langle u,v\rangle.
    \\&\text{(2) \quad} [Fu,Fv]=F[u,v].
     \\&\text{(3) \quad} \pi_{TM}\circ F=f_{*}\circ \pi_{TM}.
\end{align*}
We say that two Courant algebroids $(E,\langle\cdot,\cdot\rangle,[\cdot,\cdot],\pi)$ and $(E',\langle \cdot,\cdot\rangle',[\cdot,\cdot]',\pi')$ are \emph{isomorphic} if there exists a bundle isomorphism $F:E\to E'$ such that 
\begin{align*}
     &\text{(1) \quad} \langle Fu,Fv \rangle'=\langle u,v\rangle.
    \\&\text{(2) \quad} [Fu,Fv]'=F[u,v].
     \\&\text{(3) \quad} \pi'\circ F= \pi.
\end{align*}
\end{defn}

By using \Cref{P1}, we can see that for any exact Courant algebroid $E$ with an isotropic splitting $\sigma$
\begin{align*}
    \Aut(E)\cong_{\sigma}\Aut(TM\oplus T^*M)_H \text{  for some closed 3-form $H$}.
\end{align*}
For convenience, we denote $\GDiff_H=\Aut(TM\oplus T^*M)_H$ and we have the following.
\begin{proposition}[\cite{GRF} Proposition 2.21]\label{P2}
Let $M$ be a smooth manifold and $H$ be a closed 3-form. We have
\begin{align*}
   \GDiff_H=\{(f,\overline{f}\circ e^B): f\in\Diff(M), B\in \Omega^2 \text{  such that  } f^*H=H-dB\}
\end{align*}
where
\begin{align*}
    \overline{f}&=\begin{pmatrix} f_{\star} & 0 \\ 0 & (f^*)^{-1} \end{pmatrix}: X+\alpha\longmapsto f_{*}X+(f^*)^{-1}(\alpha),
    \\ e^B&=\begin{pmatrix} Id & 0 \\ B & Id \end{pmatrix}: X+\alpha\longmapsto X+\alpha+i_XB  \qquad \text{  for any $X\in TM$ and $\alpha\in T^*M$}.
\end{align*}
The product of automorphisms is given by
\begin{align*}
    (f,F)\circ (f',F')=\overline{f\circ f'}\circ e^{B'+f'^*B} \quad \text{  where } F=\overline{f}\circ e^B,\quad F'=\overline{f'}\circ e^{B'}.
\end{align*}
\end{proposition}

Moreover, it is proved that the equivalence class of exact Courant algebroids up to isomorphism are parametrized by $[H]\in H^3(M,\mathbb{R})$ which is known as the Ševera class of the Courant algebroid (see \cite{GRF} chapter 2.2). Next, we define the generalized metric on an exact Courant algebroid.

\begin{defn}\label{D5}
Given a smooth manifold $M$ and an exact Courant algebroid $E$ over $M$, a \emph{generalized metric} on $E$ is a bundle endomorphism $\mathcal{G}\in \Gamma(\End(E))$ satisfying 
\begin{align*}
         &\text{(1) \quad} \langle \mathcal{G}a,\mathcal{G}b \rangle=\langle a,b\rangle.
    \\&\text{(2) \quad} \langle \mathcal{G}a,b \rangle=\langle a,\mathcal{G}b\rangle.
     \\&\text{(3) \quad} \langle \mathcal{G}a,b \rangle \text{  is symmetric and positive definite for any $a,b\in E$.}
\end{align*}
\end{defn}

Following the notation in \cite{Rubio_2019}, let us denote 
\begin{itemize}[leftmargin=6cm]
\setlength\itemsep{-0.4em}
    \item[$\mathcal{GM}$:] the space of all generalized metrics.
    \item[$\mathcal{M}$:] the space of all Riemannian metrics.
    \item[$\Lambda$:] the space of all isotropic splittings.
\end{itemize}

Recall that \Cref{P1} tells us the space of generalized metrics on $E$ is isomorphic to the space of generalized metrics on $(TM\oplus T^*M)_H$ for some closed 3-form $H$. Thus, we have the following
\begin{proposition}[\cite{GRF} Proposition 2.38 and 2.40] \label{P3}
Let $E$ be an exact Courant algebroid. The space of all generalized metrics $\mathcal{GM}$ is isomorphic to $\mathcal{M}\times \Omega^2$.
\end{proposition}

\begin{remark}\label{R1}
Fix a background 3-form $H_0$ such that $E\cong (TM\oplus T^*M)_{H_0}$, the proof of \Cref{P3} implies that the 3-form $H$ of any generalized metric $\mathcal{G}=\mathcal{G}(g,b)$ is induced by an isotropic splitting $\sigma(X)=X+i_X b$. Hence, using formula (\ref{6}) we have
\begin{align*}
    H(X,Y,Z)&=2\langle [\sigma X,\sigma Y]_{H_0},\sigma Z \rangle
    \\&=2\langle [X,Y]+L_{X}b(Y)-i_{Y}(db(X))+i_Yi_X H_0, Z+i_Z b \rangle
    \\&=H_0(X,Y,Z)+(L_X b(Y)-i_Y(db(X)))(Z)+b(Z,[X,Y])
    \\&=H_0(X,Y,Z)+db(X,Y,Z),
\end{align*}
i.e., $H=H_0+db$.
\end{remark}

\subsection{Generalized Slice theorem}
In this subsection, we want to review the generalized slice theorem which was done in \cite{Rubio_2019}. First, let us briefly review the ILH manifold.
\begin{defn}\label{D6}
A set of complete locally convex topological vector spaces $\{E,E^k:k\in \mathbb{N}_{\geq d}\}$ is called an \emph{ILH chain} if for each $k$
\\(1) $E^k$ is a Hilbert space.
\\(2) $E^{k+1}$ embeds continuously in $E^k$ with dense image.
\\(3) $E=\bigcap_{k\in\mathbb{N}_{\geq d}}E^k$ is endowed with the inverse limit topology.

\end{defn}

\begin{defn}\label{D7}
Let $M$ be a manifold modeled on a locally convex topological vector space $E$. $M$ is called an \emph{ILH manifold} modelled on the ILH chain $\{E,E^k:k\in \mathbb{N}_{\geq d}\}$ if 
\\(1) The manifold $M$ is the inverse limit of smooth Hilbert manifold $M^k$ modelled on $E^k$ such that $M^l\subset M^k$ for all $l> k$.
\\(2) For all $x\in M$, there exists open charts $(U_k,\phi_k)$ of $M^k$ containing $x$ such that $U_l\subset U_k$ and $\phi_k|_{U_l}=\phi_l$ for all $l>k$. 
\end{defn}
One can also define ILH groups and ILH actions. See \cite{Rubio_2019} for more detail. 

\begin{example}
Let $(E,h)$ be a smooth Riemannian vector bundle over a compact Riemannian manifold $(M,g)$. Suppose $(E,h)$ admits a metric compatible connection $\nabla$, and we define Hilbert norms $\|\cdot\|_k$ by
\begin{align*}
    \|u\|_k=\Bigl(\sum_{i=0}^k \int_M \langle \nabla^iu,\nabla^iu \rangle_h dV_g\Bigr)^{1/2}
\end{align*}
where $\langle\cdot,\cdot\rangle_h$ is the metric on $T^*M^{\otimes l}\otimes E$ induced by $g$ and $h$ and $\nabla^i$ is the $i$-th covariant derivative. By using the Sobolev embedding, then $\{\Gamma(S^2M)\times \Omega^2,\Gamma(S^2M)^k\times \Omega^{2,k}:k\geq n+5\}$ is an ILH chain where $\Gamma(S^2M)^k\times \Omega^{2,k}$ is a completion of $\Gamma(S^2M)\times \Omega^2$ with respect to the norm $\|\cdot\|_k$.
\end{example}

Let $E$ be an exact Courant algebroid and $\mathcal{G}$ is a generalized metric. By fixing an isotropic splitting $\sigma$, we see that 
\begin{align*}
    E\cong_{\sigma}(TM\oplus T^*M)_{H}, \quad \mathcal{GM}\cong_{\sigma} \mathcal{M}\times \Omega^2, \quad \Aut(E)\cong_{\sigma}  \GDiff_H.
\end{align*}
Thus, we may write $\mathcal{G}=\mathcal{G}(g,b)$ where $(g,b)\in \mathcal{M}\times \Omega^2$. Recall that 
\begin{align*}
    \GDiff_H=\{(\varphi,B)\in\Diff(M)\ltimes\Omega^2: \varphi^*H=H-dB\},
\end{align*}
and we observe that its Lie algebra is
\begin{align*}
      \mathfrak{gdiff}_H=\{(X,\omega)\in \Gamma(TM)\times\Omega^2 : d(i_XH+\omega)=0\}.
\end{align*}
In particular, we define the following Lie subalgebra 
\begin{align*}
       \mathfrak{gdiff}_H^e=\{(X,\omega)\in \Gamma(TM)\times\Omega^2 : i_XH+\omega=d\alpha \text{  for some $\alpha\in\Omega^1$}\}
\end{align*}
and denote its integration by $\GDiff^e_H$. By Hodge theory, $\mathfrak{gdiff}_H^e$ can be viewed as the space of $(X,\omega)$ such that $i_XH+\omega$ is orthogonal to the harmonic 2-form for the fixed metric $g$.

Next, we consider an action
\begin{align}
    \rho_{\mathcal{GM}}:\nonumber\quad &(\Diff(M)\ltimes \Omega^2)\times \mathcal{GM} \longrightarrow \mathcal{GM}
    \\&((\phi,B),(g,b))\longmapsto (\phi^*g,\phi^*b-B). \label{7}
\end{align}

Based on the above notation, the following is proved in \cite{Rubio_2019} Chapter 2,3,4.
\begin{lemma}\label{L2}
Let $E$ be an exact Courant algebroid and $\mathcal{GM}$ is the space of generalized metrics. The following statements are true.
\\(1) $\mathcal{GM}$ is an ILH manifold modelled on $\{\Gamma(S^2M)\times \Omega^2,\Gamma(S^2M)^k\times \Omega^{2,k}:k\geq n+5\}$.
\\(2) $\GDiff_H$ and $ \GDiff^e_H$ are strong ILH Lie subgroups of $\Diff(M)\ltimes\Omega^2$.
\\(3) If we restrict the action $\rho_{\mathcal{GM}}$ to $\GDiff_H$ \emph{($\GDiff^e_H$)}, it will induce an \emph{ILH $\GDiff_H$ ($\GDiff^e_H$)} action on $\mathcal{GM}$. 
\end{lemma}

\begin{defn}\label{D8}
The group of \emph{generalized isometries} (exact generalized isometries) of $\mathcal{G}\in\mathcal{GM}$ is the isotropy group of $\mathcal{G}$ under the $\GDiff_H$ ($\GDiff^e_H$)-action and is denoted by $\Isom_H(\mathcal{G})$ ($\Isom^e_H(\mathcal{G})$).  
\end{defn}
\begin{remark}\label{R2}
Fix a background 3-form $H_0$. Let $\mathcal{G}=\mathcal{G}(g,b)$. \Cref{R1} indicates that the generalized metric $\mathcal{G}(g,b)$ corresponds to the 3-form $H=H_0+db$. Then, $\GDiff$-action $\rho_{\mathcal{GM}}$ on $\mathcal{G}(g,b)$ is given by
\begin{align*}
    (f,B)\cdot(g,b)\longmapsto (f^*g,f^*b-B) \quad \text{  where $f^*H_0=H_0-dB$.}
\end{align*}
If we denote the 3-form corresponding to $(f^*g,f^*b-B)$ by $H'$, we can see that
\begin{align*}
    H'=H_0+d(f^*b-B)=H_0+f^*(db)-dB=f^*(H_0+db)=f^*H.
\end{align*}
\end{remark}

To construct the generalized slice, we will need the $\GDiff_H$-invariant weak Riemannian metric on $\mathcal{GM}$ which is given by

\begin{align}
    \Big((h_1,k_1),(h_2,k_2)\Big)=\int_M \langle h_1,h_2 \rangle_g+\langle k_1,k_2 \rangle_g dV_g \qquad \text{where  } (h_1,k_1),(h_2,k_2)\in \Gamma(S^2M)\times \Omega^2. \label{8}
\end{align}
Here we note that the Riemannian metric $g$ induces a standard inner product on $\Gamma(S^2M)$ and $\Omega^2$. In \cite{Rubio_2019}, Rubio and Tipler generalized the slice theorem to the generalized metrics under the $\GDiff$ or $\GDiff^e$ action with respect to the weak Riemannian metric on $\mathcal{GM}$ (\ref{8}). However, in our case, we need to consider the $f$-twisted inner product
\begin{align}
    \Big((h_1,k_1),(h_2,k_2)\Big)_{f}\coloneqq\int_M \big(\langle h_1,h_2 \rangle_g+\langle k_1,k_2 \rangle_g \big) e^{-f} dV_g.\label{9}
\end{align}
In the following, we prove the generalized slice theorem to the $f$-twisted inner product case, and our idea is motivated by \cite{MR3319966}. 

\begin{theorem}\label{C2}
Let $\mathcal{G}$ be a generalized metric on an exact Courant algebroid $E$ and $f$ be $\Isom(\mathcal{G})$ invariant, then there exists an ILH submanifold $S^f_\mathcal{G}$ of $\mathcal{GM}$ such that 
\\(1)  $\forall F\in \Isom_H(\mathcal{G}),\, F\cdot S^f_\mathcal{G}=S^f_\mathcal{G}$.
\\(2) $ \forall F\in \GDiff_H, \text{if $(F\cdot S^f_\mathcal{G})\cap S^f_\mathcal{G}=\emptyset $ then }F\in \Isom_H(\mathcal{G})$.
\\(3) There exists a local cross section $\chi$ of the map $F\longmapsto \rho_{\mathcal{GM}}(F,\mathcal{G})$ on a neighborhood $\mathcal{U}$ of $\mathcal{G}$ in $\GDiff_H\cdot\mathcal{G}$ such that the map from $\mathcal{U}\times S^f_{\mathcal{G}}\longrightarrow \mathcal{GM}$ given by $(V_1,V_2)\longmapsto \rho_{\mathcal{GM}}(\chi(V_1),V_2)$ is a homeomorphism onto its image.   
\end{theorem}

\begin{proof}
Fix $\mathcal{G}\in\mathcal{GM}$, we define the orbit $\mathcal{O}_{\mathcal{G}}:=\GDiff_H\cdot \mathcal{G}$ and $\mathcal{O}^e_{\mathcal{G}}:=\GDiff_H^e\cdot \mathcal{G}$ with respect to the $\rho_{\mathcal{GM}}$ action. By using Proposition 4.8, 4.9 and the proof of Theorem 4.6, 4.7 in \cite{Rubio_2019}, we see that
$\mathcal{O}_{\mathcal{G}}$ and $\mathcal{O}^e_{\mathcal{G}}$ are closed ILH submanifolds of $\mathcal{GM}$. Define operators
\begin{align*}
    \textbf{A}: \nonumber \quad &\Gamma(TM\oplus T^*M)\longrightarrow \Gamma(S^2M)\times\Omega^2
    \\&X+\alpha \longmapsto (L_Xg,L_Xb+i_XH+d\alpha) 
\end{align*}
and
\begin{align*}
    \textbf{A}': \nonumber \quad &\mathfrak{gdiff}_H\longrightarrow \Gamma(S^2M)\times\Omega^2
    \\& (X,\omega) \longmapsto (L_Xg,L_Xb-\omega).
\end{align*}
We claim that $T_\mathcal{G}\mathcal{O}^e_{\mathcal{G}}=\im\textbf{A}$ and $T_\mathcal{G}\mathcal{O}_{\mathcal{G}}=\im\textbf{A}'$. Suppose $(f_t,B_t)\in\GDiff^e_H$ and $\{f_t\}$ are diffeomorphisms which are generated by the vector field $X$, i.e., $\frac{d}{dt}\Big|_{t=0} f_t=X$. We write
\begin{align*}  
    \mathcal{G}_t=\rho_{\mathcal{GM}}((f_t,B_t),(g,b))=(f_t^*g,f_t^*b-B_t)
\end{align*}
then
\begin{align*}
    \frac{d}{dt}\Big|_{t=0} \mathcal{G}_t=(L_Xg,L_Xb-\frac{d}{dt}\Big|_{t=0}B_t).
\end{align*}
Since $(f_t,B_t)\in\GDiff^e_H$, 
\begin{align*}
    i_XH+\frac{d}{dt}\Big|_{t=0}B_t=d\alpha  \text{  for some $\alpha$}.
\end{align*}
Similar to the other case and the claim is verified. 

Next, let us construct normal bundles $\nu^e=\nu^e(\mathcal{O}^e_{\mathcal{G}})$ and $\nu=\nu(\mathcal{O}_{\mathcal{G}})$ in detail. We denote the formal adjoints of $\textbf{A}$ and $\textbf{A}'$ by $\textbf{A}_f^*$ and $(\textbf{A}'_f)^*$ with respect to the $f$-twisted inner product $(,)_f$ respectively. Define 
\begin{align*}
    \nu^e_{\mathcal{G}'}=\eta^*(\ker \textbf{A}_f^*) \text{     where $\mathcal{G}'=\eta^*\mathcal{G}$ and $\eta\in\GDiff_H^e$.}
\end{align*}
Note that $f$ is $\Isom^e_\mathcal{G}$-invariant, so it is well-defined and then $\nu^e=\bigcup_{\mathcal{G}'\in \mathcal{O}_{\mathcal{G}}} \nu^e_{\mathcal{G}'}$. Following the same idea, we can define 
\begin{align*}
    \nu_{\mathcal{G}'}=\eta^*(\ker (\textbf{A}'_f)^*) \text{    where $\mathcal{G}'=\eta^*\mathcal{G}$ and $\eta\in\GDiff_H$.}
\end{align*}
and $\nu=\bigcup_{\mathcal{G}'\in \mathcal{O}_{\mathcal{G}}} \nu_{\mathcal{G}'}$ is well-defined.

Since $(,)_f$ is a weak Riemannian metric, we can not directly prove that $\nu^e$ is a smooth subbundle. Following the idea of Ebin \cite{MR0267604}, we need to construct a smooth map $P$ such that $\nu^e=\ker P$. To do so, we define 
\begin{align*}
      \textbf{B}: \nonumber \quad &\Omega^0\longrightarrow \Gamma(TM\oplus T^*M)
    \\&\beta \longmapsto (0,d\beta)
\end{align*}
and consider a sequence
\[  \begin{tikzcd}
\Omega^0\arrow[r, "\textbf{B}"] & \Gamma(TM\oplus T^*M) \arrow[r, "\textbf{A}"] & \Gamma(S^2M)\times\Omega^2.
\end{tikzcd} 
\]
with its corresponding Green operator $\mathbb{G}$ associated to $\textbf{A}_f^*\textbf{A}+\textbf{B}\textbf{B}^*_f$. Then, we define 
\begin{align*}
    P: T\mathcal{GM}|_{\mathcal{O}_\mathcal{G}}\longrightarrow T\mathcal{O}_\mathcal{G} \text{ by transporting along the orbit of the map $\textbf{A}\circ \mathbb{G}\circ \textbf{A}^*_f$.}
\end{align*}
By direct computation, 
\begin{align*}
    \Big(\textbf{A}(X,\alpha),(h,k)\Big)_f&=\int_M \Big( \langle L_Xg,h \rangle_g+\langle L_Xb+i_XH+d\alpha,k \rangle_g \Big) e^{-f}dV_g
    \\&=\int_M \Big( \langle X,-2\divg_fh+F_1(d^*k)+F_2(k) \rangle_g+2\langle \alpha,d^*_fk \rangle_g \Big) e^{-f}dV_g
\end{align*}
where $F_1$ and $F_2$ are rational functions of $(g,H,f,b,db)$ and $\divg_f$, $d^*_f$ are given in (\ref{18}). Thus, the formal adjoint $\textbf{A}^*_f$ is
\begin{align*}
    \textbf{A}^*_f(h,k)=(-2\divg_fh+F_1(d^*k)+F_2(k),\frac{1}{2}d^*_fk).
\end{align*}
Similar to the proof of Theorem 7.1 in \cite{MR0267604}, the operator $\textbf{A}^*_f$ is smooth by computing its local expression. Besides, smoothness of $\textbf{A}^*_f$ implies that $P$ is smooth and then $\nu^e=\ker P$ is smooth by the proof of Theorem 4.6 step 2 in \cite{Rubio_2019}. 

From the construction above, we can transport the operator $\textbf{A}\circ \mathbb{G}\circ \textbf{A}^*_f$ equivariantly along the $\GDiff_H$-orbit to derive a smooth ILH bundle $\nu'$ along the orbit. Denote $\mathcal{H}^2$ as the space of harmonic 2-forms, we see that
\begin{align*}
    \im \textbf{A}'=\im \textbf{A}\oplus\widetilde{\mathcal{H}}^2 \qquad \text{ where $\widetilde{\mathcal{H}}^2=\{(0,N): \text{$N\in \mathcal{H}^2$ }\} $}.  
\end{align*}
Since
\begin{align*}
    T_{\mathcal{G}}\mathcal{GM}= \im \textbf{A}\oplus\ker \textbf{A}_f^*=\im \textbf{A}'\oplus \ker (\textbf{A}'_f)^*,
\end{align*}
we have 
\begin{align*}
    \ker \textbf{A}_f^*= \ker (\textbf{A}'_f)^*\oplus \widetilde{\mathcal{H}}.
\end{align*}
That is, $\ker (\textbf{A}'_f)^*$ is the subspace of $ \ker \textbf{A}_f^*$ which is orthogonal to harmonic 2-forms. Let $\{e_i\}_{i=1}^s$ ($s=\dim \mathcal{H}^2$) be a basis of harmonic 2-forms. We can consider the projection
\begin{align*}
     \textbf{I}_0: \nonumber \quad & T_{\mathcal{G}}\mathcal{GM}\longrightarrow \mathcal{R}^s
    \\&(h,k) \longmapsto \Big((h,k),(0,e_i)\Big)_f.
\end{align*}
Extend equivariantly the operator $\textbf{I}_0$ along the orbit $\mathcal{O}_{\mathcal{G}}$, we have a smooth ILH bundle homeomorphism $\widetilde{\textbf{I}}_0$ from $\nu'$ to $\mathbb{R}^s$ over $\mathcal{O}_{\mathcal{G}}$. Then, $\nu=\ker \widetilde{\textbf{I}}_0$ is smooth.

In the remaining part of the proof, we need to use the invariant exponential map $\text{exp}_{\mathcal{GM}}$ with respect to the weak metric (\ref{8}) and a small neighborhood $\mathcal{V}$ of zero in the fiber $\nu_{\mathcal{G}}$ to construct a slice $S^f_{\mathcal{G}}$. One can follow the same argument in the third step of the proof of Theorem 4.6 in \cite{Rubio_2019} to complete the proof.
\end{proof}

\begin{remark}\label{R3}
In this paper, we may assume that our generalized metric $\mathcal{G}$ is $\mathcal{G}(g,0)$ ($b=0$). Then, 
\begin{align}
    \nonumber\ker \textbf{A}_f^*&=\Big\{(h,K)\in \Gamma(S^2M)\times\Omega^2:(\divg_f h)_l=\frac{1}{2}K^{ij}H_{ijl},\, d^*_fK=0\Big\};
    \\ \ker (\textbf{A}'_f)^*&=\Big\{(h,K)\in \Gamma(S^2M)\times\Omega^2:(\divg_f h)_l=\frac{1}{2}K^{ij}H_{ijl},\, d^*_fK=0,\, K\perp \mathcal{H}^2 \Big\}. \label{10}
\end{align}
If $f\equiv 0$, using Hodge theory we can write 
\begin{align*}
    K=K_0+d\alpha+d^*\beta \quad\text{  where $K_0\in\mathcal{H}^2$, $\alpha\in\Omega^1$ and $\beta\in\Omega^3$}
\end{align*}
so 
\begin{align}  
    \ker (\textbf{A}'_0)^*=\Big\{(h,K)\in \Gamma(S^2M)\times\Omega^2:(\divg h)_l=\frac{1}{2}K^{ij}H_{ijl},\, K=d^*\beta \Big\}.\label{11}
\end{align}
\end{remark}

\subsection{Generalized Ricci Flow and Generalized Ricci soliton}

In this section, we aim to discuss the generalized Ricci flow and generalized Ricci solitons. Most of the contents can be found in \cite{GRF} Section 4.
 
\begin{defn}\label{D9}
Let $(M,g,H)$ be a Riemannian manifold and $H\in\Omega^3$. The \emph{Bismut connections} $\nabla^{\pm}$ associated to $(g,H)$ are defined as 
\begin{align}
    \langle \nabla_X^{\pm}Y,Z \rangle=\langle \nabla_XY,Z \rangle\pm \frac{1}{2}H(X,Y,Z) \quad \text{ for all tangent vectors $X,Y,Z$.}\label{12}
\end{align}
Here, $\nabla$ is the Levi-Civita connection associated to $g$, i.e., $ \nabla^{\pm}$ are the unique compatible connections with torsion $\pm H$.
\end{defn}

Following the definitions, we are able to compute the curvature tensor concerning the Bismut connections.
\begin{proposition}[\cite{GRF} Proposition 3.18]\label{P4}
Let $(M^n,g,H)$ be a Riemannian manifold with $H\in \Omega^3$ and $dH=0$, then for any vector fields $X,Y,Z,W$ we have
\begin{align*}
    \nonumber Rm^+(X,Y,Z,W)=&Rm(X,Y,Z,W)+\frac{1}{2}\nabla_XH(Y,Z,W)-\frac{1}{2}\nabla_YH(X,Z,W)
    \\&-\frac{1}{4}\langle H(X,W),H(Y,Z)\rangle+\frac{1}{4}\langle H(Y,W),H(X,Z)\rangle, 
\end{align*}
\begin{align*}
    \Rc^+=\Rc-\frac{1}{4}H^2-\frac{1}{2}d^*H, \quad  \Rc^-=\Rc-\frac{1}{4}H^2+\frac{1}{2}d^*H, 
\end{align*}
\begin{align*}
      R^+=R-\frac{1}{4}|H|^2, \quad    R^-=R-\frac{1}{4}|H|^2 
\end{align*}
where $H^2(X,Y)=\langle i_XH,i_YH\rangle$.
\end{proposition}

\begin{defn}\label{D10}
Let $E$ be an exact Courant algebroid over a smooth manifold $M$ and $H_0$ is a background closed 3-form. A one-parameter family of generalized metrics $\mathcal{G}_t=\mathcal{G}_t(g_t,b_t)$
is called a \emph{generalized Ricci flow} if 
\begin{align}
   &\frac{\partial}{\partial t}g=-2\Rc+\frac{1}{2}H^2,\nonumber
   \\&\frac{\partial}{\partial t}b=-d^*H \quad \text{ where $H=H_0+db$.} \label{13}
\end{align}
\end{defn}

In section 5, we will consider the gauge-fixed generalized Ricci flow which can be defined as follows.
\begin{defn}\label{D11}
Let $E$ be an exact Courant algebroid over a smooth manifold $M$. A one-parameter family of generalized metrics $\{\mathcal{G}_t\}$
is called a \emph{gauge-fixed generalized Ricci flow} if there exists $F_t\in \Aut(E)$ such that $\widetilde{\mathcal{G}_t}=(F_t)_*\mathcal{G}_t$ is a solution of a generalized Ricci flow.

\end{defn}

Recall that by fixing an isotropic splitting $\sigma_0$, we have
\begin{align*}
    \Aut(E)\cong_{\sigma_0}\text{GDiff$_{H_0}$=}\{(\varphi,B)\in\Diff(M)\times\Omega^2: \varphi^*H_0=H_0-dB\}
\end{align*}
and 
\begin{align*}
      \mathfrak{gdiff}_{H_0}=\{(X,B)\in \Gamma(TM)\times\Omega^2 : d(i_XH_0+B)=0\}
\end{align*}
where $H_0$ is the 3-form corresponding to $\sigma_0$ by (\ref{6}). Thus, we can also say that $\mathcal{G}_t(g_t,b_t)$ is a gauge-fixed generalized Ricci flow if there exists $(X_t,B_t)\in  \mathfrak{gdiff}_{H_0}$ such that
\begin{align*}
  &\frac{\partial}{\partial t}g=-2\Rc+\frac{1}{2}H^2+L_{X_t}g,
  \\&\frac{\partial}{\partial t}b=-d^*H-B_t+L_{X_t}b \quad \text{where $H=H_0+db$.}
\end{align*}
In particular, let us define a one-parameter family of closed 2-form $\{k_t\}$ by
\begin{align*}
    k_t=-B_t-i_{X_t}H_0+d(i_{X_t}b).
\end{align*}
Then, we also say that $\mathcal{G}_t=\mathcal{G}_t(g_t,b_t)$ is a solution of a gauge-fixed generalized Ricci flow, if there exists $(X_t,k_t)$ such that 
\begin{align}
    &\nonumber\frac{\partial}{\partial t}g=-2\Rc+\frac{1}{2}H^2+L_{X_t}g,
     \\&\frac{\partial}{\partial t}b=-d^*H+i_{X_t}H+k_t  \quad \text{where $H=H_0+db$.} \label{14}
\end{align}

Motivated by the Ricci flow, we define the stationary points to be the steady generalized Ricci solitons.
\begin{defn}\label{D12}
Given a Riemannian metric $g$ and a closed 3-form $H$ on a smooth manifold $M$. We say $(M,g,H)$ is a \emph{steady generalized Ricci soliton} if there exists some vector field $X$ and a closed 2-form $k$ such that
\begin{align*}
    &\nonumber 0=\Rc-\frac{1}{4}H^2+\frac{1}{2}L_Xg, 
    \\&0=\frac{1}{2}d_g^*H+\frac{1}{2}i_{X}H+k. 
\end{align*}
In particular, we say $(M,g,H,f)$ is a \emph{steady gradient generalized Ricci soliton} if $X=\nabla f$ for some smooth function $f$. 
\end{defn}

In this work, we only consider the special case that $(X,k)=(\nabla f,0)$. Thus, in the following we say that $(M,g,H,f)$ is called a \emph{steady gradient generalized Ricci soliton} if
\begin{align}
    0=\Rc-\frac{1}{4}H^2+\nabla^2 f, \quad 0=d_g^*H+i_{\nabla f}H \label{15}
\end{align}
and $(M,g,H)$ is called a \emph{generalized Einstein manifold} if
\begin{align}
  0=\Rc-\frac{1}{4}H^2, \quad 0=d_g^*H.\label{16}
\end{align}

At the end of this subsection, let us write down some properties of generalized Einstein manifolds.
\begin{lemma}\label{L3}
Let $(M,g,H)$ be a generalized Einstein manifold, then $R=\frac{1}{4}|H|^2=\text{constant}$
\end{lemma}

\begin{proof}
First equality comes from the definition directly. To prove that $R$ and $|H|^2$ are both constant, we fix a normal coordinate at $p\in M$. Using the Bianchi identity, we compute
\begin{align*}
    \nabla_i R=2\nabla_j R_{ij}=\frac{1}{2}\nabla_j H^2_{ij}=\frac{1}{2}(\nabla_j H_{ikl})H_{jkl}  \quad \text{since $d^*H=0$.} 
\end{align*}
Because $H$ is closed,
\begin{align*}
    0=(dH)_{jikl}=\nabla_j H_{ikl}-\nabla_i H_{jkl}+\nabla_k H_{jil}-\nabla_l H_{jik}.
\end{align*}

The above equation reduces to 
\begin{align*}
    \nabla_i R&=\frac{1}{2}(\nabla_i H_{jkl}-\nabla_k H_{jil}+\nabla_l H_{jik})H_{jkl}
    \\&=\frac{1}{4}\nabla_i |H|^2-\nabla_k H_{jil}H_{jkl}
    \\&=\frac{1}{4}\nabla_i |H|^2-2\nabla_i R
\end{align*}
so
\begin{align*}
    \nabla_i(3R-\frac{1}{4}|H|^2)=0 \text{   for all $i$}\Longrightarrow 3R-\frac{1}{4}|H|^2=\text{constant.}
\end{align*}
Since $R-\frac{1}{4}|H|^2=0$, we complete our proof.

\end{proof}
In dimension 3, it deduces the following corollary.
\begin{corollary}\label{C3}
Any 3-dimensional generalized Einstein manifold $(M,g,H)$ has constant, nonnegative sectional curvatures. Moreover, it is Bismut flat and Einstein. If $(M,g,H)$ is compact, then it is a quotient of $S^3$.
\end{corollary}  
\begin{proof}
In the dimension 3 case, we may write $H= \phi dV_g$ where $\phi\in C^\infty (M)$. Suppose $(M,g,H)$ is a generalized Einstein manifold, then \Cref{L3} implies that $\phi\equiv C$ is a constant. Moreover,
\begin{align*}
    \Rc=\frac{1}{4}H^2=\frac{C^2}{2} g.
\end{align*}
Therefore, $(M,g,H)$ is an Einstein manifold with nonnegative sectional curvatures and it is Bismut flat by using \Cref{P4}.
\end{proof}

\section{Generalized Einstein--Hilbert Functional}
\begin{defn}\label{D13}
Let $E$ be an exact Courant algebroid over a smooth manifold $M$. The \emph{generalized Einstein--Hilbert functional} $\mathcal{F}: \Gamma(S^2M)\times \Omega^3 \times C^\infty (M)\to \mathbb{R}$ is given by 
\begin{align*}
    \mathcal{F}(g,H,f)=\int_M (R-\frac{1}{12}|H|^2+|\nabla f|^2)e^{-f}dV_g 
\end{align*}
and we define 
\begin{align*}
  \lambda(g,H)=\inf\Big\{\mathcal{F}(g,H,f)\big|\, f\in C^\infty(M),\,\int_M e^{-f}dV_g=1\Big \}. 
\end{align*}
\end{defn}

In this section, we will compute the first and second variation formula of $\lambda$. To begin with, let us consider the following notation.
\begin{align}
   (\mathring{R} h)_{jk}= R_{ijkl}h_{il}, \quad \divg \omega=\nabla_i\omega_i, \quad (\divg h)_i=\nabla_jh_{ji} \label{17}
\end{align}
and
\begin{align}
    \nonumber&\divg_f\omega=\nabla_i\omega_i-\nabla_i f\omega_i,\quad (\divg_fh)_i=\nabla_jh_{ji}-\nabla_jf h_{ji},
    \quad (\divg_f^*\omega)_{ij}=-\frac{1}{2}(\nabla_i\omega_j+\nabla_j\omega_i)=(\divg^*\omega)_{ij},
    \nonumber\\& \triangle_f=\triangle-\nabla f\cdot\nabla, \quad (d^*_fk)_j=(d^*k)_j+\nabla_i fk_{ij}, \label{18} \text{ where $\omega\in\Omega^1$, $k\in\Omega^2$ and $h\in \Gamma(S^2M)$.}
\end{align}

\subsection{First Variation Formula}

Before we discuss our main result, for completeness, we start to discuss the first variation formula of $\lambda$ which is derived in \cite{GRF}. 
\begin{lemma}[\cite{GRF} Lemma 6.7]\label{L5}
Let $(g(t),H(t),f(t))$ be a smooth family in $\mathcal{M}\times \Omega^3\times C^\infty(M)$ on a smooth compact manifold $M$. Assume 
\begin{align*}
    \frac{\partial }{\partial t}g=h, \quad \frac{\partial }{\partial t}H=dK, \quad \frac{\partial }{\partial t}f=\phi 
\end{align*}
where $(h,K,\phi)\in \Gamma(S^2M)\times\Omega^2\times C^\infty(M)$.  Then, the first variation formula of $\mathcal{F}$-functional is given by
\begin{align}
      \frac{\partial }{\partial t}\mathcal{F}(g,H,f)&\nonumber=\int_M\Big[\langle -\Rc+\frac{1}{4}H^2-\nabla^2f,h \rangle-\frac{1}{2}\langle d^*H+i_{\nabla f}H,K\rangle\Big]e^{-f}dV_g
      \\&\kern2em+\int_M (R-\frac{1}{12}|H|^2+2\triangle f-|\nabla f|^2)(\frac{\tr_gh}{2}-\phi)e^{-f}dV_g. \label{21}
\end{align}
\end{lemma}

Following the same argument in the Ricci flow case (see \cite{MR2302600} for more details), we can deduce that for any $(g,H)$, the minimizer $f$ is always achieved. Moreover, $\lambda$ satisfies that
\begin{align}
    \lambda(g,H)=R-\frac{1}{12}|H|^2+2\triangle f-|\nabla f|^2 \label{22}
\end{align}
and it is the lowest eigenvalue of the Schrödinger operator $-4\triangle+R-\frac{1}{12}|H|^2.$ Then, we have

\begin{theorem}\label{T6}
Suppose $(g_t,H_t)$ are a one-parameter family such that 
\begin{align*}
    &\frac{\partial}{\partial t}\Big|_{t=0}g_t =h, \quad  g_0=g,
    \\&  \frac{\partial}{\partial t}\Big|_{t=0}H_t =dK, \quad  H_0=H
\end{align*}
and $f_t$ satisfies that $\lambda(g_t,H_t)=\mathcal{F}(g_t,H_t,f_t)$. Then 
\begin{align}
    \frac{d}{dt}\Big|_{t=0}\lambda(g_t,H_t)=\int_M \Big[\langle -\Rc+\frac{1}{4}H^2-\nabla^2f,h \rangle-\frac{1}{2}\langle d^*H+i_{\nabla f}H,K\rangle \Big]e^{-f}dV_g \label{23}
\end{align}
where $\lambda(g,H)=\mathcal{F}(g,H,f)$.
\end{theorem}
\begin{corollary}[\cite{GRF} Proposition 6.8]\label{C4}
Suppose that $(g_t,H_t)$ is a solution of generalized Ricci flow, then 
\begin{align*}
    \frac{d}{dt}\lambda(g_t,H_t)=\int_M \Big[2|\Rc_t-\frac{1}{4}H_t^2+\nabla^2f_t|^2+\frac{1}{2}|d^*H_t+i_{\nabla f_t}H_t|^2\Big]e^{-f_t}dV_{g_t}\geq 0  \text{  where $\lambda(g_t,H_t)=\mathcal{F}(g_t,H_t,f_t)$}. 
\end{align*}
Therefore, $\lambda$ is strictly monotone increasing and $\frac{d}{dt}\lambda=0$ when $(g_t,H_t)$ is a compact steady gradient generalized Ricci soliton.

\end{corollary}

\subsection{Second Variation Formula}

By \Cref{T6}, we know that steady gradient generalized Ricci solitons are critical points of $\lambda$. In the following, we will consider a one-parameter family $(g_t,H_t,f_t)$ which starts with a compact steady gradient generalized Ricci soliton $(M,g,H,f)$ and compute the second variation formula of $\lambda$. Being motivated by \cite{MR3430881} and \cite{C2}, we need the following two lemmas.
\begin{lemma}\label{L6}
Given a compact steady gradient generalized Ricci soliton $(M,g,H,f)$. Suppose $(g_t,H_t,f_t)$ is a one-parameter family such that 
\begin{align*}
    &\frac{\partial}{\partial t}\Big|_{t=0}g_t =h, \quad  g_0=g,
    \\&  \frac{\partial}{\partial t}\Big|_{t=0}H_t =dK, \quad  H_0=H,
    \\& \frac{\partial}{\partial t}\Big|_{t=0}f_t =\phi, \quad  f_0=f
\end{align*}
and $\lambda(g_t,H_t)=\mathcal{F}(g_t,H_t,f_t)$. We have
\begin{align*}
    \nonumber&\kern-1em \frac{\partial}{\partial t}\Big|_{t=0} \Big(R_{ij}-\frac{1}{4} H^2_{ij}+\nabla_i\nabla_j f\Big)+ (\divg_f^* \divg_fh)_{ij}
    \\&= -(\mathring{R} h)_{ij}-\frac{1}{2}\triangle_f h_{ij}-\nabla_i\nabla_j(\frac{\tr_g h}{2}-\phi)+\frac{1}{8}(h_{jk}H^2_{ik}+h_{ik}H^2_{jk})+\frac{1}{2}h_{ac}H_{iab}H_{jcb}
    \\& \kern2em-\frac{1}{4}\Big((dK)_{iab}H_{jab}+H_{iab}(dK)_{jab}\Big).
\end{align*}

\end{lemma}  

\begin{proof}
For convenience, we denote $\frac{\partial}{\partial t}\Big|_{t=0}$ by $\delta$ and we compute the variation of $\Rc,H^2,\nabla^2f$, respectively.
\begin{align*}
     \delta R_{ij}&=-\frac{1}{2}\triangle_L h_{ij}-(\divg^*\divg h)_{ij}-\frac{1}{2}\nabla_i\nabla_j(\text{tr}_gh)
     \\&=-\frac{1}{2}\triangle h_{ij}-(\mathring{R} h)_{ij}+\frac{1}{2}(R_{ik}h_{kj}+h_{ij}R_{kj})-(\divg^*\divg h)_{ij}-\frac{1}{2}\nabla_i\nabla_j(\text{tr}_gh).
 \\\delta H^2_{ij}&=-2h_{ac}H_{iab}H_{jcb}+(d K)_{iab}H_{jab}+(d K)_{jab}H_{iab}.
    \\\delta(\nabla_i\nabla_j f)&=\nabla_i\nabla_j(\phi)- \frac{1}{2}(\nabla_i h_{jk}+\nabla_j h_{ik}-\nabla_k h_{ij})\nabla_k f .
\end{align*}
Also, $\divg_f^* \divg_fh$ is given by  
\begin{align*}
    (\divg_f^* \divg_fh)_{ij}&=-\frac{1}{2}(\nabla_i(\divg_fh)_j+\nabla_j(\divg_fh)_i)
    \\&=-\frac{1}{2}(\nabla_i(\nabla_k h_{jk}-h_{jk}\nabla_k f)+\nabla_j(\nabla_k h_{ik}-h_{ik}\nabla_k f))
    \\&=-\frac{1}{2}(\nabla_i\nabla_k h_{jk}+\nabla_j\nabla_k h_{ik}-h_{jk}\nabla_i\nabla_k f-h_{ik}\nabla_j\nabla_k f-\nabla_ih_{jk}\nabla_k   f-\nabla_jh_{ik}\nabla_k f)
    \\&=(\divg^*\divg h)_{ij}-\frac{1}{2}\Big(h_{jk}(R_{ik}-\frac{1}{4}H^2_{ik})+h_{ik}(R_{jk}-\frac{1}{4}H^2_{jk})-\nabla_ih_{jk}\nabla_k f-\nabla_jh_{ik}\nabla_k f\Big).
\end{align*}
Since $\Rc-\frac{1}{4}H^2+\nabla^2 f=0$ on the steady gradient generalized Ricci soliton, we complete our proof.
\end{proof}

\begin{lemma}\label{L7}
Given a compact steady gradient generalized Ricci soliton $(M,g,H,f)$, suppose $(g_t,H_t,f_t)$ is a one-parameter family such that 
\begin{align*}
    &\frac{\partial}{\partial t}\Big|_{t=0}g_t =h, \quad  g_0=g,
    \\&  \frac{\partial}{\partial t}\Big|_{t=0}H_t =dK, \quad  H_0=H,
    \\& \frac{\partial}{\partial t}\Big|_{t=0}f_t =\phi, \quad  f_0=f
\end{align*}
and $\lambda(g_t,H_t)=\mathcal{F}(g_t,H_t,f_t)$. We have
\begin{align*}
    \triangle_f(\tr_g h-2\phi)=\divg_f \divg_fh-\frac{1}{6}\langle dK,H \rangle. 
\end{align*}
\end{lemma}  
\begin{proof}
Because the first variation of $\lambda$ vanishes on a steady gradient generalized Ricci soliton, by taking the derivative of (\ref{22}), we derive
\begin{align*}
    0=\delta \lambda=\delta R-\frac{1}{12}\delta|H|^2+2\delta\triangle f-\delta|\nabla f|^2.
\end{align*}
Also, we compute
\begin{align*}
    \delta R&=-h_{ij}R_{ij}+\nabla_i\nabla_k h_{ik}-\triangle\text{tr}_gh=\frac{-1}{4}h_{ij}H^2_{ij}+h_{ij}\nabla_i\nabla_j f+\nabla_i\nabla_k h_{ik}-\triangle\text{tr}_gh,
    \\\delta |H|^2&=-3h_{ij}H^2_{ij}+2(dK)_{ijk}h_{ijk},
    \\\delta(\triangle f)&=-h_{ij}\nabla_i\nabla_j f+g^{ij}\delta(\nabla_i\nabla_j f)=-h_{ij}\nabla_i\nabla_j f+\triangle(\phi)-\nabla_i h_{ik}\nabla_k f+\frac{1}{2}\nabla_k \text{tr}_gh\nabla_k f,
    \\\delta(|\nabla f|^2)&=-h_{ij}\nabla_i f\nabla_j f+2\nabla_if\nabla_i\phi.
\end{align*}
Therefore, we see that 
\begin{align*}
 0=-\triangle(\text{tr}_gh-2\phi)+\langle \nabla(\text{tr}_gh-2\phi),\nabla f \rangle-\frac{1}{6}(dK)_{ijk}H_{ijk}+\divg_f \divg_fh.
\end{align*}
Finally,
\begin{align*}
    \divg_f \divg_fh&=\nabla_i(\nabla_jh_{ij}-h_{ij}\nabla_j f)-\nabla_i f(\nabla_jh_{ij}-h_{ij}\nabla_j f)
    \\&=\nabla_i\nabla_j h_{ij}-h_{ij}\nabla_i\nabla_j f-2\nabla_i f\nabla_jh_{ij}+h_{ij}\nabla_i f\nabla_j f
\end{align*}
and we complete the proof.
\end{proof}

Now, we are able to compute the second variation of $\lambda$.
\begin{theorem}\label{T7}
Suppose $(g_t,H_t)$ is a one-parameter family such that 
\begin{align*}
    &\frac{\partial}{\partial t}\Big|_{t=0}g_t =h, \quad  g_0=g,
    \\&  \frac{\partial}{\partial t}\Big|_{t=0}H_t =dK, \quad  H_0=H
\end{align*}
and $f_t$ satisfies that $\lambda(g_t,H_t)=\mathcal{F}(g_t,H_t,f_t)$.
The second variation of $\lambda$ on a compact steady gradient generalized Ricci soliton $(M,g,H,f)$ is given by
\begin{align}
     \nonumber\frac{d^2}{dt^2}\Big|_{t=0}\lambda&=\int_M \Big\langle \frac{1}{2}\triangle_f h+(\mathring{R} h)+\divg^*_f \divg_fh,h \Big\rangle e^{-f}dV_g
   \nonumber \\&\kern2em + \int_M \Big[-\frac{1}{4}h_{ij}h_{ik}H^2_{jk}-\frac{1}{2}h_{ij}h_{ac}H_{iab}H_{jcb}+h_{ij}(dK)_{iab}H_{jab}\Big]e^{-f}dV_g
      \nonumber\\&\kern2em -\frac{1}{6}\int_M |dK|^2 e^{-f}dV_g-\int_M \frac{1}{2}|\nabla v_{h,K}|^2e^{-f}dV_g, \label{24}
\end{align} 
where $v_{h,K}$ is the unique (up to a constant) solution of $\triangle_f(v_{h,K})=\divg_f \divg_fh-\frac{1}{6}\langle dK,H \rangle$.
\end{theorem}

\begin{proof}
First, we recall the first variation formula (\ref{23})
\begin{align*}
    \frac{d}{dt}\Big|_{t=0}\lambda=\int_M\Big[\langle -\Rc+\frac{1}{4}H^2-\nabla^2f,h \rangle-\frac{1}{2}\langle d^*H+i_{\nabla f}H,K\rangle\Big]e^{-f}dV_g.
\end{align*}
Since $(g,H,f)$ is a gradient generalized Ricci soliton, we see that the second variation is 
\begin{align*}
    \frac{d^2}{dt^2}\Big|_{t=0}\lambda&= \int_M\Big[\langle -\delta \Rc+\frac{1}{4}\delta H^2-\delta \nabla^2f,h \rangle-\frac{1}{2}\langle \delta(d^*H+i_{\nabla f}H),K\rangle\Big]e^{-f}dV_g.
\end{align*}
By \Cref{L6}, 
\begin{align*}
    &\kern-1em \int_M\langle -\delta \Rc+\frac{1}{4}\delta H^2-\delta \nabla^2f,h \rangle e^{-f}dV_g
    \\&=\int_M \langle \frac{1}{2}\triangle_f h+(\mathring{R} h)+\divg^*_f \divg_fh+\nabla^2(\frac{\text{tr}_gh}{2}-\phi),h \rangle e^{-f}dV_g
    \\&\kern2em+ \int_M \Big[-\frac{1}{4}h_{ij}h_{ik}H^2_{jk}-\frac{1}{2}h_{ij}h_{ac}H_{iab}H_{jcb}+\frac{1}{2}h_{ij}(dK)_{iab}H_{jab}\Big]e^{-f}dV_g.
\end{align*}
Next, we have
\begin{align*}
    \delta (d^* H)_{ij}&=h_{ml}\nabla_l H_{mij}-\nabla_l(dK_{lij})+(\nabla_l h_{lp}-\frac{1}{2}\nabla_p h_{ll})H_{pij}
    \\& \kern2em +\frac{1}{2}(\nabla_l h_{ip}+\nabla_i h_{lp}-\nabla_p h_{li})H_{lpj}+\frac{1}{2}(\nabla_l h_{jp}+\nabla_j h_{lp}-\nabla_p h_{lj})H_{lip}.
\end{align*}
Then, we compute
\begin{align*}
    &\kern-1em \int_M K_{ij}\delta (d^* H)_{ij} e^{-f}dV_g
    \\&=\int_M K_{ij}\Big[h_{ml}\nabla_l H_{mij}-\nabla_l(dK_{lij})+(\nabla_l h_{lp}-\frac{1}{2}\nabla_p h_{ll})H_{pij}+2K_{ij}\nabla_l h_{ip}H_{lpj}\Big]e^{-f}dV_g
    \\&=\int_M \Big[-\nabla_l(K_{ij}e^{-f})h_{ml}H_{mij}-K_{ij}\nabla_l(dK)_{lij}e^{-f}-\frac{1}{2}(\nabla_p\text{tr}_gh)K_{ij}H_{pij}e^{-f}-2h_{ip}\nabla_l K_{ij}H_{lpj}e^{-f}\Big]dV_g
    \\&=\int_M \Big[-(dK)_{lij}h_{ml}H_{mij}+K_{ij}h_{ml}H_{mij}\nabla_l f+
    \frac{1}{3}|dK|^2-K_{ij}(dK)_{lij}\nabla_lf+\frac{1}{6}\text{tr}_gh \langle dK,H \rangle\Big] e^{-f}dV_g
\end{align*}
where we use the fact that $d^*H+i_{\nabla f}H=0$. On the other hand,
\begin{align*}
    \delta(i_{\nabla f} H)_{ij}=\nabla_l\phi H_{lij}+\nabla_lf (dK)_{lij}-h_{lk}H_{kij}\nabla_lf.
\end{align*}
Hence,
\begin{align*}
   \int_M K_{ij}\delta(i_{\nabla f} H)_{ij}e^{-f}dV_g&=-\int_M K_{ij}h_{lk}H_{kij}\nabla_lfe^{-f}dV_g +\int_M K_{ij}[\nabla_l\phi H_{lij}+\nabla_l f (dK)_{lij}]e^{-f}dV_g
   \\&=-\int_M K_{ij}h_{lk}H_{kij}\nabla_lfe^{-f}dV_g-\frac{1}{3}\int_M \langle dK,H \rangle\phi e^{-f}dV_g+\int_M (dK)_{lij}K_{ij}\nabla_l fe^{-f}dV_g
\end{align*}
and
\begin{align*}
    &\kern-1em -\frac{1}{2}\int_M K_{ij}\delta \Big((d^* H)_{ij}+\delta(i_{\nabla f} H)_{ij}\Big) e^{-f}dV_g
    \\&=\frac{1}{2}\int_M h_{kl}H_{kij}(dK)_{lij}e^{-f}dV_g-\frac{1}{6}\int_M |dK|^2 e^{-f}dV_g-\frac{1}{6} \int_M \langle dK,H \rangle(\frac{\text{tr}_gh}{2}-\phi)e^{-f}dV_g.
\end{align*}

Let $v_{h,K}=\text{tr}_gh-2\phi$. \Cref{L7} implies that 
\begin{align*}
    \triangle_f(v_{h,K})=\divg_f \divg_fh-\frac{1}{6}\langle dK,H \rangle
\end{align*}
so we may write
\begin{align*}
   -\frac{1}{6} \int_M \langle dK,H \rangle(\frac{\text{tr}_gh}{2}-\phi)e^{-f}dV_g&= \int_M \Big( \triangle_f(v_{h,K})-\divg_f \divg_fh\Big)(\frac{v_{h,K}}{2})e^{-f}dV_g
   \\&=-\int_M \frac{1}{2}|\nabla v_{h,K}|^2e^{-f}dV_g-\int_M \frac{1}{2}\langle \nabla^2 v_{h,K},h \rangle e^{-f}dV_g.
\end{align*}

We remain to check the uniqueness of $v_{h,K}$. Suppose $ \triangle_f v_{h,K}=0$ then
\begin{align*}
    0=\int_M v_{h,K}  \triangle_f v_{h,K} e^{-f}dV_g= \int_M -|\nabla v_{h,K}|^2e^{-f}dV_g
\end{align*}
which implies that $v_{h,K}$ is unique up to a constant.
\end{proof}

Next, we want to find an operator to represent the second variation formula. Before we start it, let us fix our $L^2$ inner product
\begin{align}
    \Big((h_1,K_1,v_{h_1,K_1}),(h_2,K_2,v_{h_2,K_2})\Big)_{f}=\int_M \Big(\langle h_1,h_2 \rangle+\langle K_1,K_2 \rangle+ v_{h_1,K_1} v_{h_2,K_2}\Big)e^{-f}dV_g \label{25}
\end{align}
where $\langle K_1,K_2 \rangle$ is the standard inner product induced by metric $g$ on $\Omega^2$.

Define operators $A: \Gamma(S^2M)\to \Gamma(S^2M)$, $B: \Omega^2\to \Omega^2$, $C:\Omega^2\to \Gamma(S^2M) $ and $D:\Gamma(S^2M)\to \Omega^2$ by
\begin{align*}
    A(h)_{ij}&\coloneqq \frac{1}{2}\triangle_f h_{ij}+(\mathring{R} h)_{ij}+(\divg^*_f \divg_fh)_{ij}-\frac{1}{8}h_{ik}H^2_{jk}-\frac{1}{8}h_{jk}H^2_{ik}-\frac{1}{2}h_{ac}H_{iab}H_{jcb},
    \\ B(K)_{ij}&\coloneqq-\frac{1}{2}(d^*dK)_{ij}-\frac{1}{2}(dK)_{lij}\nabla_l f,
    \\ C(K)_{ij}&\coloneqq\frac{1}{4}\Big((dK)_{iab}H_{jab}+H_{iab}(dK)_{jab}\Big),
    \\ D(h)_{ij}&\coloneqq\frac{1}{2}\Big(-h_{ab}\nabla_a H_{bij}-(\divg h)_b H_{bij}+\nabla_a h_{ib}H_{baj}+\nabla_a h_{jb}H_{bia}+h_{ab}H_{bij}\nabla_a f \Big)
\end{align*}
and we define 
\begin{align}
    \nonumber\mathcal{N}:\quad &\Gamma(S^2M)\times \Omega^2\times C^\infty(M)  \longrightarrow \Gamma(S^2M)\times \Omega^2 \times C^\infty(M)
    \\&  \quad (h,K,v_{h,K})\longmapsto (A(h)+C(K),B(K)+D(h),\frac{\triangle_f v_{h,K}}{2}). \label{26}
\end{align}

\begin{corollary}\label{C5}
Define $\mathcal{N}$ as above, the second variation of $\lambda$ on a compact steady gradient generalized Ricci soliton $(M^n,g,H,f)$ is given by
\begin{align*}
   \frac{d^2}{dt^2}\Big|_{t=0}\lambda=\int_M \langle \mathcal{N}(h,K,v_{h,K}),(h,K,v_{h,K}) \rangle e^{-f}dV_g=\Big(\mathcal{N}(h,K,v_{h,K}),(h,K,v_{h,K})\Big)_{f}
\end{align*}
where $\Big(\cdot , \cdot \Big)_{f}$ is a $L^2$ inner product given by \emph{(\ref{25})}. Moreover, $\mathcal{N}$ is self-adjoint.
\end{corollary}
\begin{proof}
It is clear that the second variation formula is given by the operator $\mathcal{N}$. Following the proof of second variation formula, we see that
\begin{align*}
    \int_M \Big(\langle C(K_1),h_2 \rangle+\langle K_2,D(h_1) \rangle\Big) e^{-f}dV_g=\int_M \Big(\langle C(K_2),h_1 \rangle+\langle K_1,D(h_2) \rangle\Big) e^{-f}dV_g \quad \text{for any $(h_1,K_1)$, $(h_2,K_2)$}.
\end{align*}
Thus, $\mathcal{N}$ is self-adjoint.
\end{proof}

\section{Linear Stability}

\subsection{Introduction}

Before we move to our case, let us consider the Einstein--Hilbert functional on a Riemannian manifold $(M,g)$ 
\begin{align*}
    S(g)=\int_M R dV_g.
\end{align*}
Most of results regarding functional $S$ are written in a series papers by Koiso \cite{MR504300,MR539597,MR539596,MR558319,MR676241,MR707349}. Let us briefly review it.

Recall that on an Einstein manifold $(M,g)$ with constant $\mu$, the volume fixing second variation of $S$ at direction $h$ is given by
\begin{align*}
    S^{''}_g(h)=\int_M \Big\langle h, \frac{1}{2}\triangle h+\divg^* \divg h+(\divg\divg h)g+\frac{1}{2}\triangle(\tr_gh)g-\frac{\mu}{2}(\tr_gh)g+\mathring{R}(h)\Big\rangle dV_g.
\end{align*}
Note that we can decompose our variational space as
\begin{align*}
    \Gamma(S^2M)=(C^\infty(M)\cdot g+\im\divg^*)\oplus (TT_g)
\end{align*}
where $TT_g=\tr^{-1}_g(0)\cap\ker\divg$. Due to the fact that the functional $S$ is diffeomorphism invariant, by using Obata's theorem (see \cite{B} for more details), it suffices to consider the second variation at any direction $h\in TT_g$. Then, we see that
an Einstein manifold $(M,g)$ is linearly stable with respect to the functional $S$ if $\triangle_E|_{TT_g}$ is nonpositive where $\triangle_E=\triangle+2\mathring{R}: \Gamma(S^2M)\longrightarrow  \Gamma(S^2M)$.

It is well-known that under some curvature conditions the Einstein manifolds are stable (See \cite{B,Kroencke2014}). In this subsection, we try to use a similar idea to derive some stable results for the general Einstein--Hilbert functional. First, let us define linear stability.
\begin{defn}\label{D14}
Let $(M,g,H,f)$ be a steady gradient generalized Ricci soliton on a smooth manifold $M$. We say that $(M,g,H,f)$ is \emph{linearly stable} if 
\begin{align*}
    \frac{d^2}{dt^2}\lambda(h,K)\leq 0 \quad \text{ for all variation $(h,K)$ at $(g,H,f)$}.
\end{align*}
\end{defn}

In this section, without further notice, we always assume $(M,g,H,f)$ to be a steady gradient generalized Ricci soliton. By choosing a background 3-form $H$ properly, we consider a generalized metric $\mathcal{G}=\mathcal{G}(g,0)$ and define 
\begin{align}
    \mathcal{V}\coloneqq \{(h,K): h=L_Xg \,,\, K=-\omega \, \text{   where $(X,\omega)\in \mathfrak{gdiff}_H$} \}=\im \textbf{A}' \label{27}
\end{align}

\begin{lemma}\label{L8}
Let $\mathcal{N}$ be the second variation operator on a steady gradient generalized Ricci soliton $(M,g,H,f)$ defined in \emph{(\ref{26})}.   Then, 
\begin{align*}
    \Big(\mathcal{N}(h,K,v_{h,K}), \cdot \Big)_{f}=0 \quad \text{for all $(h,K)\in\mathcal{V}$}.
\end{align*}
\end{lemma}
\begin{proof}
Suppose $(h,K)\in\mathcal{V}$, there exists a family generalized diffeomorphisms $(f_t,B_t)$ and generalized metrics
\begin{align*}
    \mathcal{G}_t=\mathcal{G}_t(g_t,b_t)
\end{align*}
such that 
\begin{align*}
    \frac{\partial }{\partial t}\Big|_{t=0}\mathcal{G}_t=(h,K)\in\mathcal{V} \quad \text{   where   }(g_t,b_t)=(f_t^* g,-B_t).
\end{align*}
Then,
\begin{align*}
    H_t=H+d(-B_t)=f_t^*H \quad  \text{ (since  } (f_t,B_t)\in \GDiff).
\end{align*}
Recall that $\lambda$ is diffeomorphism invariant, which implies $\lambda(g_t,H_t)=\lambda(g_0,H_0)$ for all $t$, so
\begin{align*}
         \Big(\mathcal{N}(h,K,v_{h,K}), \cdot \Big)_{f}=0 \quad \text{for all $(h,K)\in\mathcal{V}$}.
\end{align*}
\end{proof}

\subsection{Gradient Ricci Soliton Case}

Now, we consider the case when $(g,H)$ is a gradient steady Ricci soliton, i.e., $H\equiv 0$. Moreover, any compact gradient Ricci steady soliton is Ricci flat and $f=0$. Thus, our second variation formula (\ref{24}) becomes 
\begin{align*}
     \nonumber\frac{d^2}{dt^2}\Big|_{t=0}\lambda&=\int_M \Big\langle \frac{1}{2}\triangle h+(\mathring{R} h)+\divg^* \divg h,h \Big\rangle dV_g
       -\frac{1}{6}\int_M |dK|^2 dV_g-\int_M \frac{1}{2}|\nabla v_{h,K}|^2dV_g.
\end{align*} 

\begin{proposition}\label{T8}
Suppose $(M,g)$ is a compact Ricci flat manifold. $(M,g)$ is linearly stable with respect to the generalized Einstein--Hilbert functional if and only if $\triangle_L\leq 0$ on $\Gamma(S^2M)$ where $\triangle_L$ is the Lichnerwoicz Laplacian.
\end{proposition}
\begin{proof}
By \Cref{L8}, it suffices to check that 
\begin{align*}
    \frac{d^2}{dt^2}\lambda(h,K)\leq 0 \quad \text{  $(h,K)\in \mathcal{V}^{\perp}$ at $(g,0)$}
\end{align*}
where $\mathcal{V}$ is defined in (\ref{27}). In this case, $H=0$, so (\ref{11}) implies that 
\begin{align*}
     \mathcal{V}^{\perp}=\ker (\textbf{A}_0')^*=\{(h,K): \divg h=0,\, K=d^*\beta, \,\beta\in\Omega^3\} \text{    and then   } v_{h,K}=\text{constant.}
\end{align*}
Finally, we compute
\begin{align*}
     \frac{d^2}{dt^2}\lambda(h,K)&=\int_M \Big\langle \frac{1}{2}\triangle h+\mathring{R}(h),h \Big\rangle-\frac{1}{6}|dK|^2 dV_g=\int_M \Big\langle \frac{1}{2}\triangle_L h,h \Big\rangle-\frac{1}{6}|dK|^2 dV_g.
\end{align*}
Thus, we complete the proof.
\end{proof}

When $(M,g)$ is Ricci flat, the Lichnerowicz Laplacian $\triangle_L$ is the same as the Einstein operator $\triangle_E$. Therefore, we have the same stability conditions and examples as the Einstein--Hilbert functional case (see \cite{B,2003math.....11253D,Kroencke2014,K1} for more details).

\subsection{General Einstein Cases}
In this subsection, we will consider a compact generalized Einstein manifold $(M^n,g,H)$ 
\begin{align*}
    \Rc-\frac{1}{4}H^2=0, \quad d^*H=0.
\end{align*}
Our second variation formula (\ref{24}) becomes
\begin{align}
     \nonumber\frac{d^2}{dt^2}\Big|_{t=0}\lambda&=\int_M \langle \frac{1}{2}\triangle_L h+\divg^* \divg h,h \rangle-\frac{1}{2}h_{ij}h_{ac}H_{iab}H_{jcb}+h_{ij}(dK)_{iab} H_{jab} dV_g
      \nonumber\\& \kern2em-\frac{1}{6}\int_M |dK|^2 dV_g-\int_M \frac{1}{2}|\nabla v_{h,K}|^2dV_g. \label{28}
\end{align} 

Before we start discussing our results, let us introduce some motivations here. Suppose $G$ is a compact Lie group, we know that $G$ possesses a bi-invariant metric $\langle\cdot ,\cdot\rangle$, and its corresponding connection, Riemann curvatures, sectional curvatures are given by
\begin{align*}
    \nabla_XY&=\frac{1}{2}[X,Y],  
    \\ R(X,Y)Z&=-\frac{1}{4}[[X,Y],Z],
    \\ K(X,Y)&=\frac{1}{4}\langle [X,Y],[X,Y] \rangle \quad \text{where $X,Y,Z$ are left-invariant vector field.}
\end{align*}

Following \cite{JM}, we choose an orthonormal basis $\{e_i\}$ for the left-invariant vector fields and we define the structure constants by
\begin{align*}
    \alpha_{ijk}=\langle[e_i,e_j],e_k \rangle.
\end{align*}
Recall that $\langle\cdot,\cdot\rangle$ is a bi-invariant metric, so
\begin{align*}
   \alpha_{ijk}=\langle[e_i,e_j],e_k \rangle=\langle e_i,[e_j,e_k]\rangle=\alpha_{jki}=\alpha_{kij}.
\end{align*}
Using the structure constants, the sectional curvatures are
\begin{align*}
    K_{ij}=\frac{1}{4}\sum_{m=1}^n \alpha^2_{ijm}\geq 0 
\end{align*}
and the Ricci curvatures are
\begin{align*}
    R_{kl}=\sum_{i=1}^n \langle R(e_i,e_k)e_l,e_i \rangle=\frac{1}{4}\sum_{i=1}^n \langle [e_i,e_k], [e_i,e_l] \rangle=-\frac{1}{4}\mathcal{B}_{kl} 
\end{align*}
where $\mathcal{B}$ is the Killing form. Using Cartan's criterion and the above discussion, we have the following proposition.

\begin{proposition}\label{P5}
A connected Lie group is compact and semisimple if and only if its Killing form is negative definite. Moreover, any compact connected semisimple Lie group, equipped with $-\mathcal{B}$ as a metric, is an Einstein manifold.
\end{proposition}

Moreover, we can define 3-form $H$ by $g^{-1}H(X,Y)=[X,Y]$ and deduce that

\begin{corollary}[\cite{GRF} Proposition 3.53]\label{C7}
A compact semisimple Lie group $G$ admits a Bismut-flat, Einstein metric $(g,H)$.
\end{corollary}

\begin{remark}\label{R8}
In \cite{AF}, the authors also proved that for any simply connected $(M,g,H)$ with flat Bismut connection, $(M,g)$ is isometric to a product of simple Lie group with bi-invariant metric $g$ and $g^{-1}H(X,Y)=\pm[X,Y]$ for any left-invariant vector fields $X,Y$ (also see \cite{GRF} Theorem 3.54).
\end{remark}

Now we start to prove our result. Recall that our trivial subspace is
\begin{align*}
    \mathcal{V}=\{(h,K): h=L_Xg \,,\, K=-\omega \, \text{   where $(X,\omega)\in \mathfrak{gdiff}_H$} \}
\end{align*}
and then we define  

\begin{align}
    \mathcal{V}_1=\begin{cases}\{(ug,K):u\in C^\infty(M), \, K\in \Omega^2\} \quad &\text{  when} \quad n\geq 4 
    \\\{(ug, -d^*(\omega dV_g) ):u,\omega\in C^\infty(M)\} \quad &\text{  when} \quad n=3 .
    \end{cases}\label{29}
\end{align}

\begin{lemma}\label{L10}
Define $\mathcal{V}$ and $\mathcal{V}_1$ as above, then 
\begin{align}
    \mathcal{V}_1^\perp=\begin{cases}\{(h,0): \tr_gh=0\} \quad &\text{  when} \quad n\geq 4
    \\\{(h, K):\tr_gh=0, \, dK=0\} \quad &\text{  when} \quad n=3
    \end{cases}\label{30}
\end{align}
and 
\begin{align}
    \mathcal{V}^\perp\cap \mathcal{V}_1^\perp=\{(h,0):\tr_gh=0,\kern0.5em\divg h=0\}\label{31}
\end{align}
with respect to the inner product \emph{(\ref{25})}.
\end{lemma}
\begin{proof}
When $n\geq 4$, it is clear that $ \mathcal{V}_1^\perp=\{(h,0):\tr_gh=0\}$. For $n=3$ case, we can write
\begin{align*}
    dK=\psi dV_g \quad \text{ where }\psi\in C^\infty(M).
\end{align*}
Then,
\begin{align*}
   (h,K)\in \mathcal{V}_1^\perp&\Longleftrightarrow \int_M \langle h,ug \rangle+\langle K, -d^*(\omega dV_g) \rangle dV_g=0   
   \\&\Longleftrightarrow \int_M u \tr_gh- \psi\omega dV_g=0 \text{  for all $u,\omega\in C^\infty(M)$}.
\end{align*}
Thus, $\tr_gh=0$ and $\psi=0$.  Finally, we recall that
\begin{align*}
   \mathcal{V}^\perp= \ker (\textbf{A}'_0)^*=\Big\{(h,K)\in \Gamma(S^2M)\times\Omega^2:(\divg h)_l=\frac{1}{2}K^{ij}H_{ijl},\, K=d^*\beta \Big\}.
\end{align*}
so it is easy to see that $\mathcal{V}^\perp\cap \mathcal{V}_1^\perp=\{(h,0):\tr_gh=0,\divg h=0\}$.
\end{proof}

In the following, we will consider the decomposition 
\begin{align}
   \Gamma(S^2M)\times \Omega^2=(\mathcal{V}+\mathcal{V}_1)\oplus(\mathcal{V}^\perp\cap \mathcal{V}_1^\perp) \label{32}
\end{align}
and check that the second variation is non-positive on each subspace under some conditions.
\begin{remark}\label{R10}
If $(M,g)$ is an Einstein manifold other than the standard sphere. The above decomposition should be 
\begin{align*}
    \Gamma(S^2M)\times \Omega^2=(\mathcal{V}\oplus\mathcal{V}_1)\oplus(\mathcal{V}^\perp\cap \mathcal{V}_1^\perp)
\end{align*}
since Obata's theorem implies that there is no nontrivial vector field $X$ such that $L_Xg$ is conformal. 
\end{remark}

\begin{lemma}\label{L11}
Let $(M,g,H)$ be a compact generalized Einstein metric. Then $\frac{d^2}{dt^2}\lambda(\mathcal{V}_1)\leq 0$ at $(M,g,H)$ where $\mathcal{V}_1$ is defined in \emph{(\ref{29}).}
\end{lemma}

\begin{proof}
Let $n\geq 4$, our second variation formula (\ref{28}) reduces to
\begin{align*}
     \frac{d^2}{dt^2}\lambda(ug, K)=\int_M \frac{n-2}{2}u\triangle u-\frac{u^2}{2}|H|^2+u\langle dK,H\rangle -\frac{1}{6}|dK|^2-\frac{1}{2}|\nabla v|^2 dV_g
\end{align*}
where $v$ is the unique solution such that $\triangle v=\triangle u-\frac{1}{6}\langle dK,H \rangle$.
Because
\begin{align*}
    -\frac{1}{6}|dK-\sqrt{3} uH|^2=-\frac{1}{6}|dK|^2+\frac{\sqrt{3} u}{3}\langle dK,H\rangle-\frac{u^2}{2}|H|^2,
\end{align*}
\begin{align*}
    \int_M -\frac{1}{2}|\nabla v-(6-2\sqrt{3}) \nabla u| dV_g&= \int_M -\frac{1}{2}|\nabla v|+(6-2\sqrt{3})\langle \nabla v,\nabla u\rangle-(24-12\sqrt{3})|\nabla u|^2 dV_g
    \\&= \int_M -\frac{1}{2}|\nabla v|-(6-2\sqrt{3}) u(\triangle u-\frac{1}{6}\langle dK,H\rangle)-(24-12\sqrt{3})|\nabla u|^2 dV_g
    \\&=\int_M -\frac{1}{2}|\nabla v|+(1-\frac{\sqrt{3}}{3})u\langle dK,H\rangle-(18-10\sqrt{3})|\nabla u|^2 dV_g,
\end{align*}
we conclude that
\begin{align*}
    \frac{d^2}{dt^2}\lambda(ug, dK)\leq \int_M -(\frac{n-2}{2}-18+10\sqrt{3})|\nabla u|^2 dV_g \leq 0 \quad \text{if $n\geq 4$.}
\end{align*}

For $n=3$ case, we may suppose $(M,g)$ is a unit sphere and then $H=2dV_g$ by \Cref{C3}. Let $dK=\psi dV_g$ where $\psi\in C^\infty(M)$. The second variation formula (\ref{28}) reduces to
\begin{align}
     \nonumber\frac{d^2}{dt^2}\Big|_{t=0}\lambda&=\int_M \Big\langle \frac{1}{2}\triangle_L h+\divg^* \divg h,h \Big\rangle-2((\tr_gh)^2-|h|^2)+4\psi(\tr_gh)-\psi^2-\frac{1}{2}|\nabla v|^2 dV_g
    \\&=\int_M \Big\langle \frac{1}{2}\triangle h+\divg^* \divg h,h \Big\rangle-((\tr_gh)^2+|h|^2)+4\psi(\tr_gh)-\psi^2-\frac{1}{2}|\nabla v|^2 dV_g.\label{33}
\end{align} 
Here, $v$ is the solution of $\triangle v=\divg\divg h-2\psi$ and we use the fact that 
\begin{align*}
    \langle \triangle_Lh,h \rangle&=\langle \triangle h,h \rangle+2R_{ijkl}h_{il}h_{jk}-2R_{ij}h_{jk}h_{ik}
    \\&=\langle \triangle h,h \rangle+4(h_{11}h_{22}+h_{11}h_{33}+h_{22}h_{33}-h_{12}^2-h_{13}^2-h_{23}^2)-4(h^2_{11}+h^2_{22}+h_{33}^2+2h^2_{12}+2h_{13}^2+2h_{23}^2)
    \\&=\langle \triangle h,h \rangle+2(\tr_gh)^2-6|h|^2.
\end{align*}
For any $(ug,-d^*(\omega dV_g))\in\mathcal{V}_1$,
\begin{align*}
    dK=\triangle\omega dV_g
\end{align*}
so
\begin{align*}
     \frac{d^2}{dt^2}\Big|_{t=0}\lambda= \int_M \frac{u\triangle u}{2}-12u^2+12u\triangle\omega-(\triangle\omega)^2-\frac{1}{2}|\nabla v|^2 dV_g
\end{align*}
where $v$ is the unique solution such that $\triangle v=\triangle u-2\triangle\omega$. We may replace $\nabla v$ by $\nabla u-2\nabla\omega$ to see that
\begin{align*}
     \frac{d^2}{dt^2}\Big|_{t=0}\lambda= \int_M -|\nabla u|^2-12u^2-10\langle \nabla u,\nabla\omega \rangle-(\triangle\omega)^2-2|\nabla \omega|^2 dV_g.
\end{align*}
Let $\mu_k$ be the eigenvalues of the Laplace operator on the unit sphere and $\chi_k$ be its corresponding eigenfunctions.
\begin{align*}
    \triangle \chi_k=-\mu_k\chi_k, \quad \mu_k=k(k+2),\quad k=0,1,2,... .
\end{align*}
Because $\chi_k$ are orthogonal basis with respect to $L^2$-inner product, we write
\begin{align*}
    u=\sum a_i \chi_i, \quad \omega=\sum b_j\chi_j \text{  where $a_i,b_j$ are all constants.}
\end{align*}
Then,
\begin{align*}
     \frac{d^2}{dt^2}\Big|_{t=0}\lambda=\sum_i \chi_i^2\Big(\int_M -(\mu_i+12)a_i^2-10a_ib_i\mu_i-(\mu_i^2+2\mu_i)b_i^2dV_g\Big).
\end{align*}
By using the Cauchy--Schwarz inequality,
\begin{align*}
    \int_M -(\mu_i+12)a_i^2-10a_ib_i\mu_i-(\mu_i^2+2\mu_i)b_i^2 dV_g \leq 0&\Longleftrightarrow |10\mu_i|\leq 2\sqrt{\mu_i+12}\sqrt{\mu_i^2+2\mu_i}
    \\&\Longleftrightarrow  0\leq \mu_i\leq 3 \text{   or   } 8\leq \mu_i.
\end{align*}
Note that $ \mu_k=k(k+2)$, so we finish the proof.
\end{proof}

\begin{lemma}\label{L12}
Suppose that $(M,g,H)$ is a Bismut-flat manifold then $\frac{d^2}{dt^2}\lambda(\mathcal{V}^\perp\cap\mathcal{V}^\perp_1)\leq 0$ at $(M,g,H)$ where $\mathcal{V}^\perp\cap\mathcal{V}_1^\perp$ is given in \emph{(\ref{31})}. In fact, $\frac{d^2}{dt^2}\lambda(h,0)\leq 0$ for all $h\in\ker\divg$.
\end{lemma}
\begin{proof}
When $(g,H)$ is a Bismut flat metric, \Cref{P4} deduces that
\begin{align*}
    Rm(X,Y,Z,W)=\frac{1}{4}\langle H(X,W),H(Y,Z)\rangle-\frac{1}{4}\langle H(Y,W),H(X,Z)\rangle\quad \text{and } \quad \nabla H=0
\end{align*}
for all vector fields $X,Y,Z,W$. In particular,
\begin{align*}
    R_{ijkl}=\frac{1}{4}H_{ilb}H_{jkb}-\frac{1}{4}H_{jlb}H_{ikb}
\end{align*}
and
\begin{align*}
    \langle\mathring{R}h,h \rangle=R_{ijkl}h_{il}h_{jk}=-\frac{1}{4}H_{jlb}H_{ikb}h_{il}h_{jk}=\frac{1}{4}H_{ljb}H_{ikb}h_{il}h_{jk}.
\end{align*}
Thus, we may rewrite our second variation (\ref{28}) 
\begin{align*}
     \frac{d^2}{dt^2}\lambda(h,0)&\leq \int_M \langle \frac{1}{2}\triangle_L h+\text{div}^* \text{div}h,h \rangle -\frac{1}{2}h_{ij}h_{ac}H_{iab}H_{jcb} dV_g
     \\&=\int_M \langle \frac{1}{2}\triangle h+\text{div}^* \text{div}h,h \rangle- \langle\mathring{R}h,h \rangle-R_{ij}h_{ik}h_{jk} dV_g.
\end{align*}
Define 
\begin{align*}
    Dh(X,Y,Z)=\frac{1}{\sqrt{3}}(\nabla_Xh(Y,Z)+\nabla_Yh(Z,X)+\nabla_Zh(X,Y)).
\end{align*}
We compute
\begin{align*}
    \|Dh\|_{L^2}^2&=\frac{1}{3}\int_M (\nabla_ih_{jk}+\nabla_jh_{ki}+\nabla_kh_{ij})^2 dV_g=\|\nabla h\|_{L^2}^2+2\int_M \nabla_ih_{jk}\nabla_jh_{ki} dV_g.
\end{align*}
Note that 
\begin{align*}
 \int_M \nabla_ih_{jk}\nabla_jh_{ki} dV_g&=-\int_M h_{jk}\nabla_i\nabla_jh_{ki} dV_g
 \\&=-\int_M h_{jk}(\nabla_j\nabla_ih_{ki}-\mathring{R}(h)_{jk}+R_{jl}h_{lk})dV_g
 \\&=\int_M |\text{div}h|^2+\langle \mathring{R}h,h\rangle- R_{jl}h_{jk}h_{lk}dV_g
\end{align*}
so we conclude that 
\begin{align}
    \|Dh\|_{L^2}^2=\|\nabla h\|_{L^2}^2+2\|\text{div} h\|_{L^2}^2+2(\mathring{R}h,h)_{L^2}-2\int_M R_{ij}h_{ik}h_{jk} dV_g. \label{34}
\end{align}
Hence, by using the fact that $\Rc=\frac{1}{4}H^2$ is positive,
\begin{align*}
    \frac{d^2}{dt^2}\lambda(h,0)&\leq \int_M \langle \frac{1}{2}\triangle h+\text{div}^* \text{div}h,h \rangle- \langle\mathring{R}h,h \rangle-R_{ij}h_{ik}h_{jk} dV_g
    \\&= -\frac{1}{2} \|Dh\|_{L^2}^2+2\|\text{div} h\|_{L^2}^2-2\int_M R_{ij}h_{ik}h_{jk} dV_g\leq 0 \quad\text{ if div $h$=0.}
\end{align*}
For $n=3$ case, $\frac{d^2}{dt^2}\lambda(\mathcal{V}^\perp\cap\mathcal{V}^\perp_1)\leq 0$ is also followed directly by (\ref{33}).

\end{proof}

Similar to the previous Einstein--Hilbert functional case, the linear stability of $\lambda$ is corresponding to a Lichnerowicz type of Laplacian. Let us define
\begin{align}
    \triangle_G h:=\frac{1}{2}\triangle h+3\mathring{R}(h)+\frac{1}{2}(\Rc\circ h+h\circ \Rc) \quad \text{ where $h\in S^2M$.} \label{35}
\end{align}

\begin{theorem}\label{P6}
Suppose that $(M,g,H)$ is Bismut-flat, Einstein manifold. $(M,g,H)$ is linearly stable if $\triangle_G|_{TT_g}$ is negative semidefinite.
\end{theorem}
\begin{proof}
For convenience, let us denote the Einstein constant by $\mu$.
By \Cref{L11} and \Cref{L12}, it suffices to check that
\begin{align*}
     \frac{d^2}{dt^2}\lambda(ug+h,K)\leq 0 \quad \text{ where $(ug,K)\in\mathcal{V}_1$ and $(h,0)\in \mathcal{V}^\perp\cap \mathcal{V}_1^\perp$, i.e., $h\in TT_g$.}
\end{align*}
Using the second variation (\ref{28}), we get
\begin{align*}
  \frac{d^2}{dt^2}\lambda(ug+h,K)&= \int_M (\frac{n-2}{2}u\triangle u-\frac{u^2}{2}|H|^2)+ \Big(\langle \frac{1}{2}\triangle h,h \rangle- \langle\mathring{R}h,h \rangle-\mu |h|^2\Big) dV_g
    \\& \kern2em+2\int_M(\frac{1}{2}\langle \triangle_L(ug),h \rangle-\frac{1}{2}(ug_{ij}h_{ac}H_{iab}H_{jcb})) dV_g
    \\& \kern2em+\int_M (ug_{ij}+h_{ij})(dK)_{iab}H_{jab}-\frac{1}{6}|dK|^2-\frac{1}{2}|\nabla v|^2 dV_g
\end{align*}
where $v$ is the unique solution of $\triangle v=\triangle u-\frac{1}{6}\langle dK,H \rangle$. First, let us compute the second line
\begin{align*}
  &\kern-1em 2\int_M\Big(\frac{1}{2}\langle \triangle_L(ug),h \rangle-\frac{1}{2}(ug_{ij}h_{ac}H_{iab}H_{jcb})\Big) dV_g
  \\&=\int_M \langle \triangle(ug),h \rangle+2R_{ijkl}(ug)_{il}h_{jk}-2R_{ij}(ug)_{ik}h_{jk}-(ug_{ij}h_{ac}H_{iab}H_{jcb}) dV_g
  \\&=\int_M (\triangle u)\tr_gh-uh_{ac}H^2_{ac} dV_g=\int_M (\triangle u-4\mu u)\tr_gh dV_g=0.
\end{align*}
Define a 3-form 
\begin{align}
    \Lambda_{aib}=\frac{1}{3}(h_{aj}H_{jib}+h_{bj}H_{jai}+h_{ij}H_{jba}) \label{36}
\end{align}
and we write 
\begin{align*}
    dK=3\Lambda + N \quad \text{where $N\in\Omega^3$. }
\end{align*}
We compute that
\begin{align*}
    \frac{1}{6}\int_M |dK|^2 dV_g&=\frac{1}{6}\int_M  \langle 3\Lambda+N,dK\rangle dV_g=\frac{1}{2}\int_M\langle \Lambda,dK\rangle dV_g+\frac{1}{6}\int_M\langle N,dK\rangle dV_g
    \\&=\frac{1}{2}\int_M\langle \Lambda,dK\rangle dV_g+\frac{1}{2}\int_M\langle N,\Lambda\rangle dV_g+\frac{1}{6}\int_M |N|^2dV_g
    \\&=\int_M\langle \Lambda,dK\rangle dV_g-\frac{3}{2}\int_M |\Lambda|^2dV_g+\frac{1}{6}\int_M |N|^2dV_g.
\end{align*}

It is a direct computation to see that 
\begin{align*}
    \langle \Lambda,H \rangle&=h_{aj}H_{jib}H_{aib}=4\mu \tr_gh=0,
    \\\langle \Lambda, dK \rangle&=h_{aj}H_{jib}(dK)_{aib},
\end{align*}
so our second variation is 
\begin{align}
     \nonumber\frac{d^2}{dt^2}\lambda(ug+h,dK)&= \int_M \frac{n-2}{2}u\triangle u-\frac{u^2}{2}|H|^2+u\langle N,H \rangle-\frac{1}{6}|N|^2-\frac{1}{2}|\nabla v|^2 dV_g
     \\&\kern2em+\int_M \langle \frac{1}{2}\triangle h,h \rangle- \langle\mathring{R}h,h \rangle-\mu |h|^2+\frac{3}{2}|\Lambda|^2 dV_g \label{37}
\end{align}
where $v$ is the unique solution of $\triangle v=\triangle u-\frac{1}{6}\langle N,H \rangle$. We observe that when $n\geq 4$ the first integral in (\ref{37}) is nonpositive by following the proof of \Cref{L11} (using the Cauchy-Schwarz inequality). In $n=3$ case,
\begin{align*}
    \Lambda_{123}=\frac{1}{3}(\tr_gh)H_{123}=0
\end{align*}
so the first integral in (\ref{37}) is also nonpositive by following the proof of \Cref{L11}.
Then, we compute
\begin{align*}
    |\Lambda^2|&=h_{aj}H_{jib}\Lambda_{aib}=\frac{1}{3}h_{aj}h_{ak}H_{jib}H_{kib}+\frac{2}{3}h_{aj}h_{bk}H_{bji}H_{kai}
    \\&=\frac{4}{3}\mu|h|^2+\frac{8}{3}\langle\mathring{R}(h),h \rangle,
\end{align*}
so
\begin{align}
    \nonumber\frac{d^2}{dt^2}\lambda(ug+h,dK)&\leq \int_M \langle \frac{1}{2}\triangle h,h \rangle- \langle\mathring{R}h,h \rangle-\mu |h|^2+\frac{3}{2}|\Lambda|^2 dV_g
    \\&=\int_M \langle \frac{1}{2}\triangle h,h \rangle+3 \langle\mathring{R}h,h \rangle+\mu |h|^2 dV_g=\int_M \langle \triangle_G h,h \rangle dV_g\leq 0 \label{38}
\end{align}
provided that $\triangle_G$ is nonpositive. 
\end{proof}

\begin{remark}\label{R11}
The converse of the statement in \Cref{P6} may not be true. The reason is that we may not be able to find a 2-form $K$ such that 
\begin{align*}
    dK=3\Lambda \quad \text{for some $h\in TT_g$.}
\end{align*}
Algebraically, the upper bound we derive is the best; however, we didn't use the fact that our variation $dK$ is not an arbitrary 3-form.

\end{remark}

\begin{corollary}\label{C8}
Any 3-dimensional generalized Einstein manifold $(M,g,H)$ is linearly stable.
\end{corollary}
\begin{proof}
For $n=3$ case, we claim that $\triangle_G|_{TT_g}$ is negative semidefinite. Due to \Cref{C3}, we write $H=2dV_g$ and compute that
\begin{align*}
    R_{ijkl}h_{il}h_{jk}&=2\Big[(h_{11}h_{22}-h^2_{12})+(h_{11}h_{33}-h^2_{13})+(h_{22}h_{33}-h^2_{23})\Big],
    \\R_{jl}h_{lk}h_{jk}&=2(h_{11}^2+h_{22}^2+h_{33}^2+2h_{12}^2+2h_{13}^2+2h_{23}^2).
\end{align*}
Then,
\begin{align*}
   &\kern-1em 3R_{ijkl}h_{il}h_{jk}+R_{jl}h_{lk}h_{jk}
   \\&=2(h_{11}^2+h_{22}^2+h_{33}^2+3h_{11}h_{22}+3h_{11}h_{33}+3h_{22}h_{33})-2(h_{12}^2+h_{13}^2+h_{23}^2)
   \\&=-(h_{11}^2+h_{22}^2+h_{33}^2)-2(h_{12}^2+h_{13}^2+h_{23}^2)=-|h|^2.
\end{align*}
Therefore,
\begin{align*}
    \int_M \langle \triangle_Gh,h \rangle dV_g&=\int_M \langle \triangle h,h \rangle+3R_{ijkl}h_{il}h_{jk}+R_{jl}h_{lk}h_{jk} dV_g\leq -\|h\|^2_{L^2}\leq 0.
\end{align*}
\end{proof}

To end this section, we would like to say that it is possible for us to find the kernel of variation.
\begin{proposition}\label{P7}
Suppose $(M,g,H)$ is a 3-dimensional generalized Einstein manifold. 
\begin{align*}
     \frac{d^2}{dt^2}\Big|_{t=0}\lambda(h,K)=0\Longleftrightarrow (h,K)\in \mathcal{V}+\widetilde{\mathcal{V}}
\end{align*}
where
\begin{align*}
    \widetilde{\mathcal{V}}=\{(ug,-d^*(\omega dV_g)): u=a_1\chi_1+a_2\chi_2,\, \omega=-a_1\chi_1-\frac{a_2}{2}\chi_2,\,  \text{  $a_1,a_2$ are constants } \} , 
\end{align*}
$\chi_1,\chi_2$ are first eignefunction and second eigenfunction of $\triangle$ respectively and $\mathcal{V}$ is defined in (\ref{27}).
\end{proposition}

\begin{proof}
By following the proof of \Cref{L11} and \Cref{P6}, we assume that $M$ is a unit sphere. In this case, we have
\begin{align*}
     \frac{d^2}{dt^2}\Big|_{t=0}\lambda\Big((ug,K)+(h,0)\Big)=\frac{d^2}{dt^2}\Big|_{t=0}\lambda(ug,K)+\frac{d^2}{dt^2}\Big|_{t=0}\lambda(h,0)=0.
\end{align*}
Here $(ug,K)\in \mathcal{V}_1$ and $(h,0)\in \mathcal{V}^\perp \cap \mathcal{V}_1^\perp$. Write $K=-d^*(\omega dV_g)$ and
\begin{align*}
    u=\sum a_i \chi_i, \quad \omega=\sum b_j\chi_j \text{  where $a_i,b_j$ are all constants, $\chi_i$ are eigenfunctions of $\triangle$.}
\end{align*}
Then,
\begin{align*}
  \frac{d^2}{dt^2}\Big|_{t=0}\lambda(ug,\triangle\omega)=0\Longleftrightarrow \int_M -(\mu_i+12)a_i^2-10a_ib_i\mu_i-(\mu_i^2+2\mu_i)b_i^2=0 \text{  for all $i$.}
\end{align*}
It implies that $a_1=-b_1$, $a_2=-2b_2$ and $(a_k,b_k)=(0,0)$ for $k\geq 3$. On the other hand,
\begin{align*}
     \frac{d^2}{dt^2}\Big|_{t=0}\lambda(h,0)= \int_M \langle \frac{1}{2}\triangle h,h \rangle-|h|^2 dV_g= 0 \Longleftrightarrow h=0.
\end{align*}
\end{proof}
\begin{remark}
In above result, we have two kernel variations $(\chi_1 g, d^*(\chi_1 dV_g))$ and $(\chi_2 g,\frac{1}{2}d^*(\chi_2 dV_g))$. In fact, the first one $(\chi_1 g, d^*(\chi_1 dV_g))\in\mathcal{V}$ is trivial and the second one can be decomposed as
\begin{align*}
    (\chi_2 g,\frac{1}{2}d^*(\chi_2 dV_g))=(\chi_2 g+\frac{1}{4}\nabla^2\chi_2,\frac{1}{4}d^*(\chi_2 dV_g))+(-\frac{1}{4}\nabla^2 \chi_2,\frac{1}{4}d^*(\chi_2 dV_g))\in \mathcal{V}^\perp+\mathcal{V},
\end{align*}
i.e., the only nontrivial kernel variation is 
\begin{align*}
    (\chi_2 g+\frac{1}{4}\nabla^2\chi_2,\frac{1}{4}d^*(\chi_2 dV_g)).
\end{align*}
\end{remark}

\section{Dynamical Stability and Instability of GRF}
The main goal of this section is to discuss dynamical stability properties. First, let us mention some definitions.

\subsection{Definition}

\begin{defn}\label{D15}
Let $(M^n,g_c,H_c,f_c)$ be a steady generalized gradient Ricci soliton with its corresponding generalized metric $\mathcal{G}_c=\mathcal{G}_c(g_c,0)$ and a background $3$-form $H_c$.
\begin{itemize}
    \item  We say that $(g_c,H_c,f_c)$ is \emph{dynamically stable} if for any neighborhood $\mathcal{U}$ of $(g_c,0)$ in $\mathcal{M}\times \Omega^2$, there exists a smaller neighborhood $\mathcal{V}\subset\mathcal{U}$ such that the generalized Ricci flow starting in $\mathcal{V}$ stays in $\mathcal{U}$ for all $t\geq 0$ and converges to a critical point $(g_\infty,b_\infty)$ of $\lambda$ with $\lambda(g_c,H_c)=\lambda(g_\infty,H_\infty)$.
    \item We say that $(g_c,H_c,f_c)$ is \emph{dynamically stable modulo diffeomorphism} if for any neighborhood $\mathcal{U}=B_\epsilon$ of $(g_c,0)$ in $\mathcal{M}\times \Omega^2$, there exists a smaller neighborhood $\mathcal{V}\subset\mathcal{U}$ such that for any generalized Ricci flow starting in $\mathcal{V}$ we have a family of diffeomorphisms $\{\varphi_t\}_{t\geq 0}$ such that
\begin{align*}
    \|(\varphi^*_tg_t,\varphi^*_tH_t)-(g_c,H_c)\|<\epsilon
\end{align*}
for all $t\geq 0$ and this GRF converges to a critical point $(g_\infty,H_\infty)$ of $\lambda$ with $\lambda(g_c,H_c)=\lambda(g_\infty,H_\infty)$.
\item We say that $(g_c,H_c,f_c)$ is \emph{dynamically unstable (modulo diffeomorphism)} if there exists a non-trivial generalized Ricci flow $(g(t),b(t))$ with $t\in(-\infty, T]$ such that $(g(t),b(t))\to (g_c,0)$ as $t\to-\infty$ (there exists a family of diffeomorphisms $\{\varphi_t\}$ with $t\in(-\infty, T]$ such that $(\varphi^*_t g_t,\varphi^*_t H_t)\to (g_c,0)$ as $t\to-\infty$).
\end{itemize}

\end{defn}

\begin{remark}
\Cref{R1} suggests us that if we fix a background closed 3-form $H_c$, we can also denote  
\begin{align*}
    \mathcal{F}(g,b,f)=\int_M (R-\frac{1}{12}|H_c+db|^2+|\nabla f|^2)e^{-f}dV_g
\end{align*}
and 
\begin{align*}
    \lambda(g,b)=\inf\Big\{\mathcal{F}(g,b,f)\big|\, f\in C^\infty(M),\,\int_M e^{-f}dV_g=1\Big \}.
\end{align*}
We will also use this notation in the remaining subsections.
\end{remark}

\begin{lemma}\label{L14}
Let $(M^n,g_c,H_c,f_c)$ be a steady generalized gradient Ricci soliton with its corresponding generalized metric $\mathcal{G}_c=\mathcal{G}_c(g_c,0)$ and a background $3$-form $H_c$. Suppose $(g_c,H_c,f_c)$ is dynamically stable or dynamically stable modulo diffeomorphism, then $(g_c,H_c,f_c)$ is a local maximum point of $\lambda$. In other words, $(g_c,H_c,f_c)$ is linearly stable.
\end{lemma}
\begin{proof}
If $(g_c,H_c,f_c)$ is not a local maximum, then for any positive number $r$ there exists a metric $(g_r,b_r)$ in the neighborhood $B_r$ of $(g_c,0)$ such that $\lambda(g_r,b_r)>\lambda(g_c,0)$. By dynamical stability, 
\begin{align*}
    (g_r,b_r)\longrightarrow (g_\infty,b_\infty) \text{ when $r$ is small enough.}
\end{align*}
However, \Cref{C5} implies that $\lambda$ is monotone increasing along the GRF, that is,
\begin{align*}
    \lambda(g_\infty,H_\infty)\geq\lambda(g_r,b_r)>\lambda(g_c,0)
\end{align*}
which is a contradiction. 
\end{proof}

\subsection{Analyticity of $\lambda$ }
In the following, we fix a generalized metric $\mathcal{G}_c=\mathcal{G}_c(g_c,0)$ and a background 3-form $H_c$. In the following, $C$ is a constant which may change from line to line.

\begin{lemma}\label{L15}
Let $\mathcal{G}_c=\mathcal{G}_c(g_c,0)$ be a generalized metric with a background 3-form $H_c$ on a compact manifold $M$. Define 
\begin{align*}
    \omega_{(g,b)}=e^{-\frac{f_{(g,b)}}{2}} \text{ where $f_{(g,b)}$ is the minimizer, i.e., $\lambda(g,b)=\mathcal{F}(g,b,f_{(g,b)})$.}
\end{align*}
Then, there exists a $C^{2,\alpha}$-neighborhood $\mathcal{U}$ of $(g_c,0)$ in $\mathcal{M}\times\Omega^2$ such that 
\begin{align*}
    \|\omega_{(g,b)}\|_{C^{2,\alpha}}\leq C .
\end{align*}
\end{lemma}
\begin{proof}
Let $(g,b)$ lie in a $C^{2,\alpha}$-neighborhood $\mathcal{U}$ of $(g_c,0)$, then
\begin{align*}
    \|g-g_c\|_{C^{2,\alpha}}<C, \quad \|b\|_{C^{2,\alpha}}<C \quad (\|H-H_c\|_{C^{1,\alpha}}=\|db\|_{C^{1,\alpha}}<C).
\end{align*}
Recall that 
\begin{align*}
    \lambda(g,b)&=\inf \left\{\int_M (R-\frac{1}{12}|H|^2+|\nabla f|^2)e^{-f}dV_g: \int_M e^{-f}dV_g=1\right\}
    \\&=\inf\left\{\int_M (R-\frac{1}{12}|H|^2)\omega^2+4|\nabla \omega|^2dV_g: \|\omega\|_{L^2}=1\right\}.
\end{align*}
By taking $\omega_{(g,b)}=e^{\frac{-f_{(g,b)}}{2}}$,
\begin{align*}
    \lambda(g,b)=\int_M (R-\frac{1}{12}|H|^2)\omega_{(g,b)}^2+4|\nabla \omega_{(g,b)}|^2dV_g.
\end{align*}
Note that if we take $\omega=(\text{Vol}(M,g))^{-\frac{1}{2}}$ then
\begin{align*}
    \lambda(g,b)\leq \int_M (R-\frac{1}{12}|H|^2)(\text{Vol}(M,g))^{-1}dV_g\leq \sup_{M}(R-\frac{1}{12}|H|^2),
\end{align*}
therefore,
\begin{align*}
    4\|\nabla\omega_{(g,b)}\|^2_{L^2}&=\lambda(g,b)-\int_M (R-\frac{1}{12}|H|^2)\omega_{(g,b)}^2 dV_g
    \\&\leq  \sup_{M}(R-\frac{1}{12}|H|^2) - \inf_{M}(R-\frac{1}{12}|H|^2)\leq C \quad \text{(since $M$ is compact)} .
\end{align*}
Then
\begin{align*}
    \|\omega_{(g,b)}\|_{W^{1,2}}\leq C&\Longrightarrow \|\omega_{(g,b)}\|_{L^{\frac{2n}{n-2}}}\leq C \quad \text{(by the Sobolev embedding theorem)}
    \\&\Longrightarrow \|\omega_{(g,b)}\|_{W^{\frac{2n}{n-2}},2}\leq C \quad \text{(by elliptic regularity).}
\end{align*}
Using Hölder's inequality, the Sobolev theorem, and elliptic regularity several times, we can see that
\begin{align*}
    \|\omega_{(g,b)}\|_{W^{p,2}}\leq C \text{   for all $p\in (1,\infty)$.}
\end{align*}
By choosing $p$ large enough, we have $\|\omega_{(g,b)}\|_{C^{2,\alpha}}\leq C$.
\end{proof}

\begin{proposition}\label{P8}
Let $(M,g_c,H_c,f_c)$ be a compact steady gradient generalized Ricci soliton with its corresponding generalized metric $\mathcal{G}_c=\mathcal{G}_c(g_c,0)$ and a background $3$-form $H_c$. There exists a $C^{2,\alpha}$-neighborhood $\mathcal{U}$ of $(g_c,0)$ such that the minimizers $f_{(g,b)}$ depends analytically on $(g,b)$ and $(g,b)\longmapsto\lambda(g,b)$ is analytic in $\mathcal{U}$.
\end{proposition}
\begin{proof}
Let $S(g,b,f)=R-\frac{1}{12}|H|^2+2\triangle f-|\nabla f|^2$ and we define
\begin{align*}
    \nonumber\mathcal{L}:\quad & C^{2,\alpha}(\mathcal{M}\times \Omega^2)\times C^{2,\alpha}(M)\longrightarrow C^{0,\alpha}_{g_c}(M)\times \mathbb{R}
    \\&((g,b),f)\longmapsto (S(g,b,f)-\fint_M S(g,b,f)e^{-f_c}dV_{g_c},\int_M e^{-f}dV_g-1) 
\end{align*}
where $C^{0,\alpha}_{g_c}(M)=\{\phi\in C^{0,\alpha}(M): \int_M \phi e^{-f_c} dV_{g_c}=0\}$. Note that $\mathcal{L}$ is an analytic map and 
\begin{align*}
    \mathcal{L}(g,b,f)=(0,0)\Longleftrightarrow S(g,b,f)=\text{constant and } \int_M e^{-f}dV_g=1.
\end{align*}

We can compute that 
\begin{align*}
    dS(h,K,\phi)&=-\triangle(\tr_{g_c}h-2\phi)+\langle\nabla(\tr_{g_c}h-2\phi),\nabla f_c\rangle-\frac{1}{6}\langle dK,H_c \rangle+\divg_{f_c}\divg_{f_c} h,
    \\ d(\int_M e^{-f_c}dV_{g_c})&=\int_M (\frac{\tr_{g_c}h}{2}-\phi)e^{-f_c}dV_{g_c} \quad \text{ at $(g_c,0,f_c)$}
\end{align*}
so 
\begin{align*}
    d\mathcal{L}(0,0,\phi)|_{(g_c,0,f_c)}=(2\triangle\phi-2\langle\nabla\phi,\nabla f_c\rangle,-\int_M \phi e^{-f_c} dV_{g_c})\in C^{0,\alpha}_{g_c}(M)\times \mathbb{R}
\end{align*}
which implies that $d\mathcal{L}(0,0,\phi)|_{(g_c,0,f_c)}$ is a linear isomorphism. By the implicit function theorem, there exists a $C^{2,\alpha}$-neighborhood $\mathcal{U}$ of $(g_c,0)$ and an analytic map $\mathcal{P}:\mathcal{U}\longrightarrow C^{2,\alpha}(M)$ such that
\begin{align*}
    \mathcal{L}((g,b),\mathcal{P}(g,b))=0.
\end{align*}
Moreover, the implicit function theorem also implies that there exists a $C^{2,\alpha}$-neighborhood $\mathcal{V}$ of $f_{c}$ such that if $\mathcal{L}(g,b,f)=0$ for some $(g,b)\in\mathcal{U}$ and $f\in\mathcal{V}$ then $f=\mathcal{P}(g,b)$.

We claim that there exists a smaller neighborhood $\mathcal{V}\subset \mathcal{U}$ such that $\mathcal{P}(g,b)=f_{(g,b)}$ for all $(g,b)\in\mathcal{V}$. Suppose the claim is false, then there exists a sequence 
\[  \begin{tikzcd}
(g_i,b_i) \arrow[r, "C^{2,\alpha}"] & (g_c,0)
\end{tikzcd} \text{  such that $f_{(g_i,b_i)}\neq \mathcal{P}(g_i,b_i)$ for all $i$.}
\]
We define $\omega_i=e^{\frac{-f_{(g_i,b_i)}}{2}}$. By \Cref{L15} and the Arzela--Ascoli theorem, there exists a convergent subsequence such that
\[  \begin{tikzcd}
\omega_i \arrow[r, "C^{2,\alpha'}"] & \omega_{\infty}
\end{tikzcd} \text{  for some $\alpha'<\alpha$.}
\]
Then,
\begin{align*}
    \lambda(g_c,0)&=\int_M (R_c-\frac{1}{12}|H_c|^2)\omega_c^2+4|\nabla\omega_c|^2 dV_{g_c}
    \\&\leq \int_M (R_c-\frac{1}{12}|H_c|^2)\omega^2_{\infty}+4|\nabla\omega_{\infty}|^2dV_{g_c}=\lim_{i\to\infty} \lambda(g_i,b_i)
    \\&\leq \lim_{i\to\infty}\mathcal{F}(g_i,b_i,f_c)=\lambda(g_c,0).
\end{align*}
Therefore, $\omega_{\infty}=\omega_c$ and then $f_{(g_i,b_i)}\longrightarrow f_c$. On the other hand, $\mathcal{L}(g_i,b_i,f_{(g_i,b_i)})=0$ for all $i$ so this contradicts the implicit function theorem and we prove the claim. By using the claim, we see that 
\begin{align*}
    f_{(g,b)}=\mathcal{P}(g,b) \text{  is analytic when $(g,b)\in\mathcal{V}$}
\end{align*}
and $\lambda(g,b)=S((g,b),\mathcal{P}(g,b))$ is analytic.
\end{proof}
  
\begin{remark}\label{R15}
We can use the same idea to show that there exists a $C^{2,\alpha}$ neighborhood $\mathcal{U}$ of $(g_c,0)$ such that for any minimizer $f_{(g,b)}$ of $(g,b)\in\mathcal{U}$, we have
\begin{align*}
    \|f_{(g,b)}\|_{C^{2,\alpha}}\leq C. 
\end{align*}
\end{remark}
\begin{proof}
We observe that it is equivalent to prove that $\omega_{(g,b)}$ is bounded away from 0. If it is not true, we have a sequence 
\[  \begin{tikzcd}
(g_i,b_i) \arrow[r, "C^{2,\alpha}"] & (g_c,0)
\end{tikzcd} \text{  such that $\min\omega_{(g_i,b_i)}=\min e^{\frac{-f_{(g_i,b_i)}}{2}}\longrightarrow 0$.}
\]
However, follow the same proof in \Cref{P8}, we see that $\omega_{(g_i,b_i)}\longrightarrow\omega_c> 0$ which is a contradiction.
\end{proof}

\subsection{Lojasiewicz--Simon Inequality for $\lambda$}

Before we start to prove the Lojasiewicz--Simon inequality, we need a lemma.

\begin{lemma}\label{L16}
Let $(M,g_c,H_c, f_c)$ be a steady gradient generalized Ricci soliton with its corresponding generalized metric $\mathcal{G}_c=\mathcal{G}_c(g_c,0)$ and a background $3$-form $H_c$. We can choose a $C^{2,\alpha}$-neighborhood $\mathcal{U}$ of $(g_c,0)$ such that 
\begin{align*}
    \nonumber&\|\frac{d}{dt}\Big|_{t=0}f_{(g+th,b+tK)}\|_{C^{2,\alpha}}\leq C(\|h\|_{C^{2,\alpha}}+\|K\|_{C^{2,\alpha}})\quad (g,b)\in\mathcal{U},
    \\&  \|\frac{d}{dt}\Big|_{t=0}f_{(g+th,b+tK)}\|_{W^{2,2}}\leq C(\|h\|_{W^{2,2}}+\|K\|_{W^{2,2}})\quad (g,b)\in\mathcal{U}.
\end{align*}
\end{lemma}

\begin{proof}
By using the first variation formula (\ref{23}), 
\begin{align*}
    \frac{d}{dt}\Big|_{t=0}\lambda(g+th,b+tK)=\int_M \Big[\langle -\Rc+\frac{1}{4}H^2-\nabla^2f,h \rangle-\frac{1}{2}\langle d^*H+i_{\nabla f}H,K\rangle \Big]e^{-f}dV_g.
\end{align*}
We can see that 
\begin{align*}
    \|\frac{d}{dt}\Big|_{t=0}\lambda(g+th,b+tK)\|_{C^{0,\alpha}}\leq C(\|h\|_{C^{2,\alpha}}+\|K\|_{C^{2,\alpha}}).
\end{align*}
On the other hand,
\begin{align*}
     \frac{\partial}{\partial t}\lambda(g+th,b+tK)&=  \frac{\partial}{\partial t}(R-\frac{1}{12}|H|^2+2\triangle f-|\nabla f|^2)
     \\&=2\triangle_f(\frac{\partial f}{\partial t})-\triangle_f(\text{tr}_g h)+\divg_f \divg_fh-\frac{1}{6}\langle dK,H \rangle.
\end{align*}
By elliptic regularity,
\begin{align*}
    \|\frac{\partial f}{\partial t}\|_{C^{2,\alpha}}&\leq C \|\triangle_f(\frac{\partial f}{\partial t})\|_{C^{0,\alpha}}
    \\&\leq C\|\frac{\partial \lambda}{\partial t}+\triangle_f(\text{tr}_g h)-\divg_f \divg_fh+\frac{1}{6}\langle dK,H \rangle\|_{C^{0,\alpha}}
    \\&\leq C(\|h\|_{C^{2,\alpha}}+\|K\|_{C^{2,\alpha}})
\end{align*}
where we use the result that $\|f\|_{C^{2,\alpha}}$ is uniformly bounded. The $W^{2,2}$-norm case follows similarly.

\end{proof}

Now, we want to prove the Lojasiewicz--Simon inequality by using \cite{MR3211041} Theorem 6.3. For completeness, let us write down their statement here.
\begin{theorem}[\cite{MR3211041} Theorem 6.3]\label{T10}
Suppose the functional $G$ satisfies the following assumptions:
\begin{enumerate}
    \item  Let $E$ be a closed subspace of $L^2$ maps to a finite-dimensional vector space and $G$ is an analytic functional defined on a neighborhood $\mathcal{O}_E$ of $0$ in $C^{2,\alpha}\cap E$.
    \item $\nabla G:\mathcal{O}_E\longrightarrow C^{0,\alpha}\cap E$ is a $C^1$ map with $\nabla G(0)=0$ and
\begin{align*}
    \|\nabla G(x)-\nabla G(y)\|_{L^2}\leq C\|x-y\|_{W^{2,2}}.
\end{align*}
\item The linearization $\mathcal{L}$ of $\nabla G$ at 0 is symmetric, bounded from $C^{2,\alpha}\cap E$ to $ C^{\alpha}\cap E$ and from $W^{2,2}\cap E $ to $ L^2\cap E$, and is Fredholm from $C^{2,\alpha}\cap E$ to $ C^{\alpha}\cap E$.

\end{enumerate}
Then, there exists a $\beta\in(0,1)$ and a neighborhood of 0 in $E$ such that
\begin{align*}
    |G(x)-G(0)|^{2-\beta}\leq \|\nabla G(x)\|^2_{L^2} \text{  for $x\in E$.}
\end{align*}
\end{theorem}

\begin{theorem}[Lojasiewicz--Simon inequality for $\lambda$ ]\label{T11}
Let $(M^n,g_c,H_c,f_c)$ be a steady generalized gradient Ricci soliton with its corresponding generalized metric $\mathcal{G}_c=\mathcal{G}_c(g_c,0)$ and a background $3$-form $H_c$. There exists a $C^{2,\alpha}$-neighborhood $\mathcal{U}$ of $(g_c,0)$ and a $\beta\in (0,1)$ such that 
\begin{align*}
    |\lambda(g,b)-\lambda(g_c,0)|^{2-\beta}\leq \|\nabla\lambda(g,H)\|_{L^2}^2=\|(\Rc-\frac{1}{4}H^2+\nabla^2f,\frac{1}{2}(d^*H+i_{\nabla f}H))\|_{L^2}^2
\end{align*}
for all $(g,b)\in\mathcal{U}$.
\end{theorem}
\begin{proof}
By \Cref{P8}, we have a $C^{2,\alpha}$ neighborhood $\mathcal{U}$ such that $\lambda(g,b)$ is analytic which satisfies the first assumption in \Cref{T10}. Next, let us compute the $L^2$-gradient. The first variation formula  (\ref{23}) says 
\begin{align*}
    \nabla\lambda|_{(g,b)}=\Big(-(\Rc-\frac{1}{4}H^2+\nabla^2f),-\frac{1}{2}(d^*H+i_{\nabla f}H)\Big). 
\end{align*}
It is clear that $\nabla\lambda$ is a $C^1$-map and $\nabla\lambda|_{(g_c,0)}=0$.
\begin{align*}
    &\kern-1em\|\nabla\lambda|_{(g_1,b_1)}-\nabla\lambda|_{(g_2,b_2)}\|_{L^2}
    \\& \leq \| (\Rc_1-\frac{1}{4}H_1^2+\nabla^2 f_1)-(\Rc_2-\frac{1}{4}H_2^2+\nabla^2 f_2)\|_{L^2}
    +\|-\frac{1}{2}(d^*H_1+i_{\nabla f_1}H_1)+\frac{1}{2}(d^*H_2+i_{\nabla f_2}H_2)\|_{L^2}
    \\&\leq C(\|g_2-g_1\|_{W^{2,2}}+\|b_2-b_1\|_{W^{2,2}})
\end{align*}
where we use the Taylor series expansion and \Cref{L16}. Finally, we check the linearization of $\nabla\lambda$ at $(g_c,0)$.

 Since $f$ is $\Isom({\mathcal{G}_c})$-invariant ( $(g_c,H_c,f_c)$ is a steady gradient generalized Ricci soliton so $f$ is invariant) and the diffeomorphism invariance of $\lambda$, it suffices to consider the $f$-twisted generalized slice $\mathcal{S}^f_{\mathcal{G}_c}\cap \mathcal{U}$. On the slice, we observe that 
\begin{align*}
    \divg_{f_c}\divg_{f_c} h&=\nabla_i(\divg_{f_c}h)_i-\langle \nabla f_c,\divg_{f_c}h\rangle=\nabla_i(\frac{1}{2}K_{ab}H_{abi})-\frac{1}{2}K_{ab}H_{abi}\nabla_if_c=\frac{1}{6}\langle dK,H\rangle  
\end{align*}
which implies that $v_{h,K}$ is a constant. Then, the operator $\mathcal{N}$ in (\ref{26}), related to the second variation formula, reduces to
\begin{align*}
      \nonumber\mathcal{N}:\quad &\Gamma(S^2M)\times \Omega^2  \longrightarrow \Gamma(S^2M)\times \Omega^2 
    \\&  \quad (h,K)\longmapsto (A(h)+C(K),B(K)+D(h)).
\end{align*} 
Here,
\textbf{\begin{align*}
    A(h)_{ij}&= \frac{1}{2}\triangle_L h_{ij}-\frac{1}{2}h_{ac}H_{iab}H_{jcb}-\frac{1}{2}h_{ik}\nabla_j\nabla_k f_c-\frac{1}{2}h_{jk}\nabla_i\nabla_k f_c-\nabla_lf_c H_{lij},
    \\ B(K)_{ij}&=\frac{1}{2}\triangle K_{ij}+\frac{1}{2}K_{jl}\nabla_i\nabla_l f_c+\frac{1}{2}K_{il}\nabla_j\nabla_lf_c-\frac{1}{2}\nabla_lK_{ij}\nabla_lf_c,
    \\ C(K)_{ij}&=\frac{1}{4}((dK)_{iab}H_{jab}+(dK)_{jab}H_{iab}),
    \\ D(h)_{ij}&=-\frac{1}{2}h_{ab}\nabla_aH_{bij}+\frac{1}{2}\nabla_ah_{ib}H_{baj}+\frac{1}{2}\nabla_ah_{jb}H_{bia}
\end{align*}}
and $(h,K)\in\ker(\textbf{A}'_{f_c})^*$ is defined in (\ref{10}).
Thus, the linearization of $\nabla\lambda$ is 
\begin{align*}
    L_{\nabla\lambda}=(A(h)+C(K),B(K)+D(h)).
\end{align*}
The second order part of $ L_{\nabla\lambda}$ is a Laplace operator so it is an elliptic operator. By elliptic regularity, $ L_{\nabla\lambda}$ is bounded and $ L_{\nabla\lambda}$ is a Fredholm operator. Also, it is symmetric by the integration by parts so $\lambda$ satisfies all assumptions in \Cref{T10} and we prove the Lojasiewicz--Simon inequality.
\end{proof}

\subsection{Estimates on GRF}

In this subsection, we discuss two key lemmas which will be used to prove the dynamical stability.

\begin{lemma}[Estimate for $t\leq 1$]\label{L17} 
Let $\mathcal{G}_c=\mathcal{G}_c(g_c,0)$ be a generalized metric on a smooth manifold $M$ with a background 3-form $H_c$. Suppose $k\geq 2$. For each $C^{k}$ neighborhood $\mathcal{U}$ of $(g_c,0)$ in $\mathcal{M}\times\Omega^2$, there exists a $C^{k+2}$ neighborhood $\mathcal{V}$ of $(g_c,0)$ such that  the generalized Ricci flow starting at any $(g,b)\in\mathcal{V}$ stays in $\mathcal{U}$ for all $t\in[0,1]$. 
\end{lemma}
\begin{proof}
First, w.l.o.g., we may suppose $\mathcal{U}=B^k_\epsilon$ is an $\epsilon$-ball of $(g_c,0)$ with respect to the $C^k_{g_c}$-norm ($C^k$-norm with the derivatives depending on the metric $g_c$).  Let $(g(t),b(t))$ be the generalized Ricci flow starting at $(g,b)$ and suppose that $T\in[0,\infty]$ be the maximal time such that $(g(t),b(t))$ exists and 
\begin{align*}
    \|(g(t)-g,b(t))\|_{C^{k}_{g_c}}<\epsilon. 
\end{align*}

We claim that we can pick $\delta$ small enough such that $\delta<\frac{\epsilon}{4}$ and 
\begin{align*}
    \|(-2\Rc(t)+\frac{1}{2}H^2(t),-d^* H(t))\|^2_{C^k_{g(t)}}<\frac{\epsilon}{8} \quad \text{  provided that  } \|(g-g_c,b)\|_{C^{k+2}_{g_c}}<\delta \text{  and  } t\leq 1.
\end{align*}
According to the assumption $\|(g(t)-g_c,b(t))\|_{C^{k}_{g_c}}<\epsilon $, we see that
\begin{align*}
    \sup|\nabla^i Rm|\text{   and   } \sup|\nabla^{i+1}H| \text{  are bounded for $i\leq k-2$.}
\end{align*}
In \cite{GRF} Section 5.3, we have the evolution equations.
\begin{align*}
    (\frac{\partial}{\partial t}-\triangle)|\nabla^l Rm|^2_{g(t)}&=-2|\nabla^{l+1}Rm|^2_{g(t)}+\nabla^{l+2}H^2*\nabla^l Rm+\sum_{j=0}^l \nabla^j(Rm+H^2)*\nabla^{l-j}Rm*\nabla^lRm,
    \\(\frac{\partial}{\partial t}-\triangle)|\nabla^l H|^2_{g(t)}&=-2|\nabla^{l+1}H|^2_{g(t)}+\sum_{j=0}^l\nabla^{j}Rm*\nabla^{l-j} H*\nabla^kH+\sum_{j=1}^l \nabla^jH^2*\nabla^{l-j}H*\nabla^lH. 
\end{align*}
By the Cauchy--Schwarz inequality, we can see that 
\begin{align*}
    (\frac{\partial}{\partial t}-\triangle)(|\nabla^{k-1} Rm(t)|^2_{g(t)}+|\nabla^kH(t)|^2_{g(t)})\leq C_1(|\nabla^{k-1} Rm(t)|^2_{g(t)}+|\nabla^kH(t)|^2_{g(t)})+C_2.
\end{align*}
The maximum principle implies  
\begin{align*}
    &|\nabla^{k-1} Rm(t)|^2_{g(t)}+|\nabla^kH(t)|^2_{g(t)}\leq \widetilde{C}_1(T,\epsilon,\|(g-g_c,b)\|_{C^{k+2}_{g_c}}).
\end{align*}
Using the same argument, we deduce that 
\begin{align*}
    |\nabla^{k} Rm(t)|^2_{g(t)}+|\nabla^{k+1}H(t)|^2_{g(t)}\leq \widetilde{C}_2(T,\epsilon,\|(g-g_c,b)\|_{C^{k+2}_{g_c}})
\end{align*}
and then we pick $\delta$ small enough such that  
$\delta<\frac{\epsilon}{4}$ and 
\begin{align*}
    \|(-2\Rc(t)+\frac{1}{2}H^2(t),-d^* H(t))\|^2_{C^k_{g(t)}}<\frac{\epsilon}{8} \text{  when $t\leq 1$.}
\end{align*}

Next, we suppose $\epsilon$ is small enough such that the $C^k$-norms with respect to $g(t)$ and $g_c$ differ at most by a factor of 2. If $T\leq 1$
\begin{align*}
    \|(g(T)-g_c,b(T))\|_{C^k_{g_c}}&\leq \|(g-g_c,b)\|_{C^k_{g_c}}+\int_0^T  \frac{d}{dt}\|(g(t)-g,b(t))\|_{C^k_{g_c}} dt
    \\&\leq\|(g-g_c,b)\|_{C^k_{g_c}}+2\int_0^T  \|(-2\Rc(t)+\frac{1}{2}H^2(t),-d^*H(t))\|_{C^k_{g(t)}} dt\leq \frac{\epsilon}{2}
\end{align*}
which is a contradiction.
\end{proof}

\begin{lemma}[\cite{GRF} Theorem 5.18. Estimate for $t\geq 1$] \label{L18}
 Suppose $(g(t),b(t))$ is a solution to the generalized Ricci flow on $[0,T)$ with $1\leq T$ satisfying 
 \begin{align*}
     \sup(|Rm|+|\nabla H|+|H|^2)\leq K \quad t\in[0,T), \text{  for some constant $K$}.
 \end{align*}
 
Given $l\in\mathbb{N}$, there exists constant $C=C(n,l,K)$ such that
\begin{align*}
    \sup(|\nabla^lRm|+|\nabla^{l+1} H|+|H||\nabla^l H|)\leq C, \quad t\in[1,T).
\end{align*}
\end{lemma}

\subsection{Dynamical Stability and Instability}

Now, we start to prove the main theorem.
\begin{theorem}\label{T12}
Let $(M^n,g_c,H_c,f_c)$ be a steady generalized gradient Ricci soliton with its corresponding generalized metric $\mathcal{G}_c=\mathcal{G}_c(g_c,0)$ and a background $3$-form $H_c$. Suppose $(M^n,g_c,H_c,f_c)$ is linearly stable and $k\geq 3$. Then, for every $C^k$-neighborhood $\mathcal{U}=B^k_\epsilon$ of $(g_c,0)$ with $\epsilon$ small enough, there exists some $C^{k+2}$-neighborhood $\mathcal{V}$ such that for any GRF $(g(t),b(t))$ starting in $\mathcal{V}$, we have a family of diffeomorphisms $\{\varphi_t\}$ satisfying
\begin{align*}
    \|(\varphi^*_tg_t,\varphi^*_tH_t)-(g_c,H_c)\|_{C^{k-1}_{g_c}}<\epsilon \quad \text{ for all $t$}
\end{align*}
and $(\varphi^*_t g_t,\varphi^*H_t)$ converges to $(g_\infty,H_\infty)$ of polynomial rate with $\lambda(g_c,H_c)=\lambda(g_\infty,H_\infty)$.
\end{theorem}
\begin{proof}
By \Cref{T11} and \Cref{L17}, we can take $\epsilon$ small enough such that in the $C^k_{g_c}$ neighborhood $\mathcal{U}=B^k_\epsilon$ of $(g_c,0)$ we have the following two properties.
\\\begin{itemize}
    \item The Lojasiewicz--Simon inequality 
\begin{align}
    |\lambda(g,H)-\lambda(g_c,H_c)|^\alpha\leq \|(\Rc-\frac{1}{4}H^2+\nabla^2f,\frac{1}{2}(d^*H+i_{\nabla f}H))\|_{L^2}  \text{ holds  for some $\alpha\in(\frac{1}{2},1)$}. \label{40}
\end{align}
\item We can choose a smaller $C^{k+2}_{g_c}$-neighborhood $\mathcal{V}$ such that the GRF starting in $\mathcal{V}$ will stay in $B^k_{\frac{\epsilon}{4}}$ up to time 1.
\end{itemize}

Assume $T\geq 1$ be the maximal time such that for any solution of GRF starting in $\mathcal{V}$, there exists a family of diffeomorphisms $\{\varphi_t\}$ such that 
\begin{align*}
    \|(\varphi^*_tg_t,\varphi^*_tH_t)-(g_c,H_c)\|_{C^{k-1}_{g_c}}<\epsilon \quad \text{ for all $t\leq T$.}
\end{align*}

Let $(g(t),b(t))$ be a solution of the generalized Ricci flow starting with $(g(0),b(0))\in\mathcal{V}$. We define the modified flow $(\widetilde{g}(t),\widetilde{b}(t))$ starting with $(g(0),b(0))$ as follows.
\begin{itemize}
    \item For $t\leq 1$: Define $(\widetilde{g}(t),\widetilde{b}(t))=(g(t),b(t))$.
    \item For $t\geq 1$: We define $(\widetilde{g}(t),\widetilde{b}(t))$ by the $(-\nabla_{g_t}f_{(g_t,H_t)},0)$-gauge fixed generalized Ricci flow (\ref{14}). More precisely, 
\begin{align*}
    &\frac{\partial \widetilde{g}}{\partial t}=-2(\widetilde{\Rc}-\frac{1}{4}\widetilde{H}^2+\nabla^2\widetilde{f}), \quad \widetilde{g}(1)=g(1),
    \\&\frac{\partial \widetilde{b}}{\partial t}=-(d^*\widetilde{H}+i_{\nabla\widetilde{f}}\widetilde{H}), \quad \widetilde{b}(1)=b(1). 
\end{align*}
\end{itemize}

Denote $\psi_t$ the diffeomorphism generated by $X(t)=-\nabla_{g_t}f_{(g_t,H_t)}$ and $\psi_1=Id$, the modified flow $(\widetilde{g},\widetilde{H})$ can also be expressed as
\begin{align}
    \widetilde{g}(t)=\begin{cases}g(t) & t\in[0,1] \\ \psi^*_t g(t) & t\geq 1
    \end{cases}, \quad \widetilde{H}(t)=\begin{cases}H(t) & t\in[0,1] \\ \psi^*_t H(t) & t\geq 1
    \end{cases}.  \label{41}
\end{align}

Let $\widetilde{T}\geq 1$ be the maximal time such that the modified flow $(\widetilde{g},\widetilde{H})$ satisfies
\begin{align*}
    \|(\widetilde{g},\widetilde{H})-(g_c,H_c)\|_{C^{k-1}_{g_c}}<\epsilon, \quad \text{ for all $t\leq \widetilde{T}$ ($\widetilde{T}\leq T$ by definition}). 
\end{align*}
Recall the Hamilton's interpolation theorem for tensor $S$ (\cite{10.4310/jdg/1214436922} Corollary 12.7), there exists a constant $C$ such that
\begin{align*}
    \int_M |\nabla^i S|^2 dV_g\leq C(\int_M |\nabla^n S|^2 dV_g)^{\frac{i}{n}}(\int_M|S|^2 dV_g)^{1-\frac{i}{n}}\quad \text{for $0\leq i\leq n$, $n=\dim M$}. 
\end{align*}
Using the interpolation theorem and \Cref{L18}, we can find a $\beta\in(0,1)$ as small as we want such that when $t\geq 1$,
\begin{align*}
    \|(\frac{\partial \widetilde{g}}{\partial t},\frac{\partial \widetilde{H}}{\partial t})\|_{C^{k-1}_{g(t)}}&\leq   \|(\frac{\partial \widetilde{g}}{\partial t},\frac{\partial \widetilde{b}}{\partial t})\|_{C^{k}_{g(t)}}
    \\&\leq C  \|(\frac{\partial \widetilde{g}}{\partial t},\frac{\partial \widetilde{b}}{\partial t})\|_{W^{l,2}_{g(t)}} \quad \text{(pick $l$ large enough and use the Sobolev embedding)}
    \\&\leq   C\|(\frac{\partial \widetilde{g}}{\partial t},\frac{\partial \widetilde{b}}{\partial t})\|^{1-\beta}_{L^2_{g(t)}}. 
\end{align*}

By picking this $\beta$ sufficiently small, we have  $\theta=1-\alpha(1+\beta)>0$.  Hence,
\begin{align*}
    &\kern-1em-\frac{d}{dt}|\lambda(\widetilde{g}(t),\widetilde{H}(t))-\lambda(g_c,H_c)|^\theta
    \\&=\theta |\lambda(\widetilde{g}(t),\widetilde{H}(t))-\lambda(g_c,H_c)|^{\theta-1}\frac{d}{dt}\lambda(\widetilde{g}(t),\widetilde{H}(t))
    \\&= \theta |\lambda(\widetilde{g}(t),\widetilde{H}(t))-\lambda(g_c,H_c)|^{-\alpha(1+\beta)}\int_M\left( 2|\widetilde{\Rc}-\frac{1}{4}\widetilde{H}^2+\nabla^2\widetilde{f}|^2+\frac{1}{2}|d^*\widetilde{H}+i_{\nabla\widetilde{f}}\widetilde{H}|^2\right) e^{-\widetilde{f}}dV_g
    \\&\geq C \|(\widetilde{\Rc}-\frac{1}{4}\widetilde{H}^2+\nabla^2\widetilde{f},\frac{1}{2}(d^*\widetilde{H}+i_{\nabla\widetilde{f}}\widetilde{H}))\|_{L^2}^{1-\beta}\geq\|(\frac{\partial \widetilde{g}}{\partial t},\frac{\partial \widetilde{H}}{\partial t})\|_{C^{k-1}_{g(t)}} 
\end{align*}
where we use the Lojasiewicz--Simon inequality (\ref{40}). Then,
\begin{align*}
    \int_1^{\widetilde{T}}\|(\frac{\partial \widetilde{g}}{\partial t},\frac{\partial \widetilde{H}}{\partial t})\|_{C^{k-1}_{g(t)}} dt&\leq C \int_1^{\widetilde{T}}( -\frac{d}{dt}|\lambda(\widetilde{g}(t),\widetilde{H}(t))-\lambda(g_c,H_c)|^\theta) dt
    \\&= C\left(|\lambda(\widetilde{g}(1),\widetilde{H}(1))-\lambda(g_c,H_c)|^\theta-|\lambda(\widetilde{g}(\widetilde{T}),\widetilde{H}(\widetilde{T}))-\lambda(g_c,H_c)|^\theta\right)
    \\&\leq C|\lambda(\widetilde{g}(0),\widetilde{H}(0))-\lambda(g_c,H_c)|^\theta <\frac{\epsilon}{4}.
\end{align*}
By shrinking $\mathcal{V}$ small enough, we may suppose that the $C^k$-norms with respect to $g(t)$ and $g_c$ differ at most by a factor 2. Then,
\begin{align*}
    \|(\widetilde{g}(\widetilde{T}),\widetilde{H}(\widetilde{T}))-(g_c,H_c)\|_{C^{k-1}_{g_c}}&\leq  \|(\widetilde{g}(1),\widetilde{H}(1))-(g_c,H_c)\|_{C^{k-1}_{g_c}}+ \int_1^{\widetilde{T}}\frac{d}{dt} \|(\widetilde{g}(t),\widetilde{H}(t))-(g_c,H_c)\|_{C^{k-1}_{g_c}}dt
    \\&\leq \frac{\epsilon}{4}+2\int_1^{\widetilde{T}}\|(\frac{\partial \widetilde{g}}{\partial t},\frac{\partial \widetilde{H}}{\partial t})\|_{C^{k-1}_{g(t)}} dt\leq \frac{3\epsilon}{4}
\end{align*}
which is a contradiction. Therefore, $\widetilde{T}=T=\infty$ (long time existence) and $(\widetilde{g},\widetilde{H})\longrightarrow (g_\infty,H_\infty)$. Using the Lojasiewicz--Simon inequality (\ref{40}) again,
\begin{align*}
    \frac{d}{dt}|\lambda(\widetilde{g}(t),\widetilde{H}(t))-\lambda(g_c,H_c)|^{1-2\alpha}&=(2\alpha-1) |\lambda(\widetilde{g}(t),\widetilde{H}(t))-\lambda(g_c,H_c)|^{-2\alpha}\frac{d}{dt}\lambda(\widetilde{g}(t),\widetilde{H}(t))
    \\&\geq C |\lambda(\widetilde{g}(t),\widetilde{H}(t))-\lambda(g_c,H_c)|^{-2\alpha}\|(\widetilde{\Rc}-\frac{1}{4}\widetilde{H}^2+\nabla^2\widetilde{f},\frac{1}{2}(d^*\widetilde{H}+i_{\nabla^2\widetilde{f}}\widetilde{H}))\|_{L^2}^{2}
    \\&\geq C.
\end{align*}
Integrating this differential inequality gives
\begin{align*}
    |\lambda(\widetilde{g}(t),\widetilde{H}(t))-\lambda(g_c,H_c)|\leq C(t+1)^{\frac{1}{1-2\alpha}}\xrightarrow[\text{ as }{t} \rightarrow\infty]{}0
\end{align*}
so $\lambda(g_{\infty},H_{\infty})=\lambda(g_c,H_c)$. Moreover, we can prove that the convergence is of a polynomial rate
\begin{align*}
    \|(\widetilde{g}(t_1),\widetilde{H}(t_1))-(\widetilde{g}(t_2),\widetilde{H}(t_2))\|_{C^{k-1}_{g_c}}&=\int_{t_1}^{t_2}\|(\frac{\partial \widetilde{g}}{\partial t},\frac{\partial \widetilde{H}}{\partial t})\|_{C^{k-1}_{g(t)}} dt
    \\&\leq C|\lambda(\widetilde{g}(t_1),\widetilde{H}(t_1))-\lambda(g_c,H_c)|^\theta\leq C(t_1+1)^{\frac{\theta}{1-2\alpha}}.
\end{align*} 
Let $t_2\to\infty$, and then
\begin{align*}
    \|(\widetilde{g}(t_1),\widetilde{H}(t_1))-(g_\infty,H_\infty)\|_{C^{k-1}_{g_c}}\leq C(t_1+1)^{\frac{\theta}{1-2\alpha}}.
\end{align*}

\end{proof}

\begin{theorem}\label{T13}
Let $(M^n,g_c,H_c,f_c)$ be a steady generalized gradient Ricci soliton with its corresponding generalized metric $\mathcal{G}_c=\mathcal{G}_c(g_c,0)$ and a background $3$-form $H_c$. Suppose $(g_c,H_c)$ is not a local maximum point of $\lambda$, then there exists an ancient flow of GRF and a family of diffeomorphism $\{\varphi_t\}$, $t\in (-\infty,0]$ such that  
\begin{align*}
    (\varphi^*_t g_t,\varphi^*_t H_t)\longrightarrow (g_c,H_c) \text{  as $t\to-\infty$.}
\end{align*}
\end{theorem}

\begin{proof}
Since $(g_c,H_c)$ is not a local maximum, there exists a sequence $(g_i,b_i)$ such that 
\[  \begin{tikzcd}
(g_i,b_i) \arrow[r, "C^{k}"] & (g_c,0)
\end{tikzcd} \text{  and $\lambda(g_i,b_i)>\lambda(g_c,0)$ for all $i$.}
\]

For any fixed $i$. we consider the modified flow $(\widetilde{g}_i(t),\widetilde{b}_i(t))$ defined similarly in (\ref{41}) starting with $(g_i,b_i)$. 
Then, \Cref{L17} implies that
\[  \begin{tikzcd}
 (\widetilde{g}_i(1),\widetilde{b}_i(1)) \arrow[r, "C^{k-2}"] & (g_c,0).
\end{tikzcd} 
\]

We can take $\epsilon$ small enough such that the Lojasiewicz--Simon inequality holds on $C^{k-2}$-neighborhood $B_{\epsilon}$ of $(g_c,0)$. Due to the assumption that $\lambda(g_c,H_c)$ is not a local maximum,
\begin{align*}
    \frac{d}{dt}\left(\lambda(\widetilde{g}_i,\widetilde{H}_i)-\lambda(g_c,H_c) \right)^{1-2\alpha}&=(1-2\alpha)\left(\lambda(\widetilde{g}_i,\widetilde{H}_i)-\lambda(g_c,H_c) \right)^{-2\alpha}\frac{d}{dt}\lambda(\widetilde{g}_i,\widetilde{H}_i)
    \\&\geq -C.
\end{align*}
By the maximum principle, 
\begin{align*}
    \nonumber&\left(\lambda(\widetilde{g}_i(t),\widetilde{H}_i(t))-\lambda(g_c,H_c) \right)^{1-2\alpha}\geq \left(\lambda(\widetilde{g}_i(s),\widetilde{H}_i(s))-\lambda(g_c,H_c) \right)^{1-2\alpha}-C(t-s)
\end{align*}
and
\begin{align}
    \left( (\lambda(\widetilde{g}_i(t),\widetilde{H}_i(t))-\lambda(g_c,H_c) )^{1-2\alpha}-C(s-t)  \right)^{\frac{1}{1-2\alpha}}\leq \lambda(\widetilde{g}_i(t),\widetilde{H}_i(t))-\lambda(g_c,H_c) \label{42}.
\end{align}

Therefore, we can pick $\{t_i\}$ ($t_i\geq 1$) such that 
\begin{align*}
    \|(\widetilde{g}_i(t_i),\widetilde{b}_i(t_i))-(g_c,0)\|_{C^{k-2}}=\epsilon \quad \text{ and $t_i\longrightarrow \infty$.} 
\end{align*}
Define
\begin{align*}
   \begin{cases}
    \widetilde{g}^s_i(t)=\widetilde{g}_i(t+t_i) \\ \widetilde{b}^s_i(t)=\widetilde{b}_i(t+t_i) 
   \end{cases}
   \quad t\in [T_i,0] \text{  where $T_i=1-t_i\longrightarrow -\infty$}.
\end{align*}
Then, 
\begin{align*}
    \|(\widetilde{g}^s_i(t),\widetilde{b}^s_i(t))-(g_c,0)\|_{C^{k-2}}\leq \epsilon \quad \text{ for all $t\in [T_i,0]$}
\end{align*}
and
\[  \begin{tikzcd}
 (\widetilde{g}^s_i(t),\widetilde{b}^s_i(t))=(\widetilde{g}_i(1),\widetilde{b}_i(1)) \arrow[r, "C^{k-2}"] & (g_c,0).
\end{tikzcd} 
\]

By compactness result, there exists a subsequence of $(\widetilde{g}_i^s,\widetilde{b}^s_i)$ converging in $C^{k-3}_{loc}(M\times (-\infty,0])$ to an ancient flow $(\widetilde{g}(t),\widetilde{b}(t))$, $t\in(-\infty,0]$. Note that $(\widetilde{g}(t),\widetilde{b}(t))$ is a solution of (\ref{41}). Next, we consider the diffeomorphism $\psi_t$ generated by $X(t)=\widetilde{\nabla}_{\widetilde{g}(t)}\widetilde{f}_{\widetilde{g_t},\widetilde{H_t}}$ then $(g(t),H(t))=(\psi^*_t\widetilde{g}(t),\psi^*_t\widetilde{H}(t))$ is a solution of GRF. Using the same idea in the proof of \Cref{T12}, we have
\begin{align*}
    \|(\widetilde{\Rc}-\frac{1}{4}\widetilde{H}^2+\nabla^2\widetilde{f},\frac{1}{2}(d^*\widetilde{H}+i_{\nabla \widetilde{f}}\widetilde{H}))\|_{C^{k-2}}\leq \|(\widetilde{\Rc}-\frac{1}{4}\widetilde{H}^2+\nabla^2\widetilde{f},\frac{1}{2}(d^*\widetilde{H}+i_{\nabla \widetilde{f}}\widetilde{H}))\|^{1-\beta}_{L^2} \text{  for some $\beta\in (0,1)$.}
\end{align*}

Let $\theta=1-\alpha(1+\beta)>0$, and we derive
\begin{align*}
    &\kern-1em\frac{d}{dt}\Big(\lambda(\widetilde{g}_i(t),\widetilde{b}_i(t))-\lambda(g_c,0)\Big)^\theta
    \\&=\theta \Big(\lambda(\widetilde{g}_i(t),\widetilde{b}_i(t))-\lambda(g_c,0)\Big)^{\theta-1}\frac{d}{dt}\lambda(\widetilde{g}_i(t),\widetilde{b}_i(t))
    \\&\geq \theta \left(\lambda(\widetilde{g}_i(t),\widetilde{b}_i(t))-\lambda(g_c,0) \right)^{-\alpha(1+\beta)}\|(\widetilde{\Rc}-\frac{1}{4}\widetilde{H}^2+\nabla^2\widetilde{f},\frac{1}{2}(d^*\widetilde{H}+i_{\nabla \widetilde{f}}\widetilde{H}))\|^2_{L^2}
    \\&\geq C\|(\widetilde{\Rc}-\frac{1}{4}\widetilde{H}^2+\nabla^2\widetilde{f},\frac{1}{2}(d^*\widetilde{H}+i_{\nabla \widetilde{f}}\widetilde{H}))\|_{C^{k-2}}.
\end{align*}
Thus,
\begin{align*}
    \epsilon&= \|(\widetilde{g}_i(t_i),\widetilde{b}_i(t_i))-(g_c,0)\|_{C^{k-2}}
    \\&\leq \|(\widetilde{g}_i(1),\widetilde{b}_i(1))-(g_c,0)\|_{C^{k-2}}+\int_1^{t_i}\frac{d}{dt}\|(\widetilde{g}_i(t),\widetilde{b}_i(t))-(g_c,0)\|_{C^{k-2}} dt
    \\&\leq \|(\widetilde{g}_i(1),\widetilde{b}_i(1))-(g_c,0)\|_{C^{k-2}}+C\int_1^{t_i} \frac{d}{dt}\Big(\lambda(\widetilde{g}_i(t),\widetilde{b}_i(t))-\lambda(g_c,0)\Big)^\theta dt
    \\&\leq \|(\widetilde{g}_i(1),\widetilde{b}_i(1))-(g_c,0)\|_{C^{k-2}}+C\Big(\lambda(\widetilde{g}_i(t_i),\widetilde{b}_i(t_i))-\lambda(g_c,0)\Big)^\theta.
\end{align*}
By taking $i\to\infty$, we see that 
\begin{align*}
    \epsilon\leq C(\lambda(g(0),b(0))-\lambda(g_c,0))^\theta, 
\end{align*}
which implies that the generalized Ricci flow is non-trivial. Finally, we check that $(\widetilde{g}(t),\widetilde{H}(t))$ will converge to $(g_c,H_c)$ as $t\to-\infty$.  For $t\in [T_i,0]$,
\begin{align*}
    &\kern-1em\|(\widetilde{g}^s_i(T_i),\widetilde{H}_i^s(T_i))-(\widetilde{g}^s_i(t),\widetilde{H}_i^s(t))\|_{C^{k-3}}
    \\&\leq \|(\widetilde{g}^s_i(T_i),\widetilde{b}_i^s(T_i))-(\widetilde{g}^s_i(t),\widetilde{b}_i^s(t))\|_{C^{k-2}}\leq \int_{T_i}^t \frac{d}{dt}\|(\widetilde{g}^s_i(t),\widetilde{b}_i^s(t))\|_{C^{k-2}}dt
    \\&\leq C \int_{T_i}^t \frac{d}{dt}\left(\lambda(\widetilde{g}^s_i(t),\widetilde{H}^s_i(t))-\lambda(g_0,H_0)\right)^\theta dt\leq C\left(\lambda(\widetilde{g}_i(t+t_i),\widetilde{H}_i(t+t_i))-\lambda(g_0,H_0)\right)^\theta .
    \\&\leq \Big(-Ct+ C(\lambda(\widetilde{g}_i(t_i),\widetilde{H}_i(t_i))^{1-2\alpha}\Big)^{\frac{\theta}{1-2\alpha}}
\end{align*}
where we use inequality (\ref{42}). We observe that
\begin{align*}
     \|(g_c,H_c)-(\widetilde{g}^s_i(T_i),\widetilde{H}_i^s(T_i))\|_{C^{k-3}}= \|(g_c,H_c)-(\widetilde{g}_i(1),\widetilde{H}_i(1))\|_{C^{k-3}}\leq\|(g_c,0)-(\widetilde{g}_i(1),\widetilde{b}_i(1))\|_{C^{k-2}}\rightarrow 0
\end{align*}
so
\begin{align*}
    &\kern-1em\|(g_c,H_c)-(\widetilde{g}(t),\widetilde{H}(t))\|_{C^{k-4}}
    \\&\leq \|(g_c,H_c)-(\widetilde{g}^s_i(T_i),\widetilde{H}_i^s(T_i))\|_{C^{k-4}}+\|(\widetilde{g}_i^s(T_i),\widetilde{H}_i^s(T_i))-(\widetilde{g}^s_i(t),\widetilde{H}_i^s(t))\|_{C^{k-4}}
    \\&\kern2em+\|(\widetilde{g}(t),\widetilde{H}(t))-(\widetilde{g}^s_i(t),\widetilde{H}_i^s(t))\|_{C^{k-4}}
\end{align*}
Take $i\to -\infty$ and $t\to-\infty$, we see that
\begin{align*}
    \|(g_c,H_c)-(\widetilde{g}(t),\widetilde{H}(t))\|_{C^{k-4}}\to 0.
\end{align*}
It implies
\begin{align*}
    (\varphi_t^{-1})^*(g_t,H_t)\longrightarrow (g_c,H_c)  \quad \text{  in $C^{k-4}$ as $t\to-\infty$}.
\end{align*}
\end{proof}

\bibliographystyle{plain}
\bibliography{Reference}
\nocite{*}

\end{document}